\documentclass{amsart}

\title{Simple zeros of $\GL(2)$ $L$-functions}

\author{Alexandre de Faveri}
\address{Stanford University, Stanford, CA 94305, USA}
\email{\href{mailto:afaveri@stanford.edu}{afaveri@stanford.edu}}

\usepackage{amssymb, amsmath, amsthm}
\usepackage{thmtools}
\usepackage{hyperref, enumitem}
\usepackage[space]{cite}
\usepackage[centering]{geometry}
\usepackage{xcolor, graphicx}
\graphicspath{ {./images/} }
\usepackage{bbm}
\usepackage{mathrsfs}
\usepackage{pgf}
\usepackage{caption}
\usepackage{subcaption}
\usepackage[utf8]{inputenc}\DeclareUnicodeCharacter{2212}{-}

\newcommand{\dd}{\, d}
\newcommand{\R}{\mathbb{R}}
\renewcommand{\H}{\mathbb{H}}
\newcommand{\C}{\mathbb{C}}

\newcommand{\Z}{\mathbb{Z}}
\newcommand{\Q}{\mathbb{Q}}
\renewcommand{\L}{\mathcal{L}}
\renewcommand{\S}{\mathcal{S}}
\newcommand{\eps}{\varepsilon}
\newcommand{\rn}{\epsilon}
\newcommand{\1}{\mathbbm{1}}

\newcommand{\Mod}[1]{\ (\mathrm{mod}\ #1)}

\DeclareMathOperator*{\Res}{Res}
\DeclareMathOperator*{\sgn}{sgn}
\DeclareMathOperator*{\GL}{GL}

\DeclareFontFamily{U}{mathx}{\hyphenchar\font45}
\DeclareFontShape{U}{mathx}{m}{n}{
      <5> <6> <7> <8> <9> <10>
      <10.95> <12> <14.4> <17.28> <20.74> <24.88>
      mathx10
      }{}
\DeclareSymbolFont{mathx}{U}{mathx}{m}{n}
\DeclareFontSubstitution{U}{mathx}{m}{n}
\DeclareMathAccent{\widecheck}{0}{mathx}{"71}
\DeclareMathAccent{\wideparen}{0}{mathx}{"75}

\newtheorem{theorem}{Theorem}
\newtheorem{lemma}[theorem]{Lemma}
\newtheorem{proposition}[theorem]{Proposition}
\newtheorem{corollary}[theorem]{Corollary}

\newtheorem{remark}[theorem]{Remark}

\numberwithin{theorem}{section}
\numberwithin{lemma}{section}
\numberwithin{proposition}{section}
\numberwithin{corollary}{section}
\numberwithin{definition}{section}
\numberwithin{property}{section}
\numberwithin{remark}{section}
\numberwithin{conjecture}{section}

\numberwithin{equation}{section}

\begin{document}

\begin{abstract}
    Let $f \in S_k(\Gamma_1(N))$ be a primitive holomorphic form of arbitrary weight $k$ and level $N$. We show that the completed $L$-function of $f$ has $\Omega\left(T^\delta\right)$ simple zeros with imaginary part in $\left[-T, T\right]$, for any $\delta < \frac{2}{27}$. This is the first power bound in this problem for $f$ of non-trivial level, where previously the best results were $\Omega(\log\log\log{T})$ for $N$ odd, due to Booker, Milinovich, and Ng, and infinitely many simple zeros for $N$ even, due to Booker. In addition, for $f$ of trivial level ($N=1$), we also improve an old result of Conrey and Ghosh on the number of simple zeros.
\end{abstract}

\maketitle


\section{Introduction}

\subsection{Discussion}

Let $\pi$ be a cuspidal automorphic representation of $\GL(n, \mathbb{A}_\Q)$ with completed $L$-function $\Lambda_\pi$. It is conjectured that all the zeros of $\Lambda_\pi(s)$ are on the critical line $\Re(s) = \frac{1}{2}$ and, apart from at most one multiple zero of algebraic origin, are all simple. For degree $n=1$ (Dirichlet $L$-functions), Levinson's method \cite{Le74, He79, Ba00, Yo10} shows that a positive proportion of the zeros are simultaneously simple and on the critical line. An adaptation of that method for degree $n=2$ also implies that a positive proportion of the zeros are on the critical line \cite{AT20}, but already cannot tackle simple zeros and only shows that a positive proportion of the zeros are of order at most three \cite{Fa94}. 

In this paper we consider the problem of obtaining lower bounds for the number of simple zeros in the case of degree $n=2$. Let $f \in S_k(\Gamma_1(N))$ be a primitive form (i.e.\ a normalized Hecke newform) of arbitrary weight $k$ and level $N$. The first challenge is to show that $\Lambda_f$ has any simple zeros at all. While for a given $f$ this can be checked computationally, the problem was only completely solved in 2012, after a breakthrough of Booker \cite{Bo16}, who in fact showed that $\Lambda_f$ has infinitely many simple zeros. The argument relies on simple zeros of local factors of $\Lambda_f$, thus differentiating it from counterexamples such as the square of a degree one $L$-function. Another key ingredient in Booker's method is non-vanishing of automorphic $L$-functions on the line $\Re(s) = 1$, more specifically applied to multiplicative twists of $f$, foreshadowing an important obstruction in the method.

With Booker's result in hand, the next challenge is to obtain quantitative bounds on the number of simple zeros of $\Lambda_f$. Here one runs into issues related to the level $N$ that are somewhat reminiscent of the difficulties in extending the Hecke converse theorem to general level. As in Weil's generalization of the converse theorem, an important tool are the twists of $f$ by multiplicative characters. However, in our case an obstruction remains. It roughly consists of the possibility that $\Lambda_f(s)$ has simple zeros arbitrarily close to the line $\Re(s) = 1$, and in addition that a certain conspiracy between additive twists of $f$ happens at those simple zeros --- namely that \eqref{delta_twist_simplified_eq} below does not have a pole at any of those simple zeros, for any choice of $\alpha \in \Q^\times$.

Let
\begin{equation*}
    N_f^s(T) := \left|\left\{\rho \in \C : |\Im(\rho)| \leq T \text{ and } \rho \text{ is a simple zero of } \Lambda_f \right\}\right|
\end{equation*}
denote the number of simple zeros of $\Lambda_f$ with imaginary part in $[-T, T]$. For the case of full level ($N=1$), it is easy to directly check that no widespread pole cancellation in \eqref{delta_twist_simplified_eq} can happen. In a paper from 1988 which introduced ideas used in most subsequent works on this topic, Conrey and Ghosh \cite{CG88} showed that if $f = \Delta$ is the Ramanujan function, then $N_f^s(T) = \Omega(T^{\frac{1}{6} - \eps})$ for any $\eps > 0$. Their method applies to any $f$ of level $N=1$, as long as one assumes the existence of at least one simple zero for $\Lambda_f$ (which they verified for $f = \Delta$, and is now known to hold in general due to Booker's work).

For general level $N$, Booker, Milinovich, and Ng \cite{BMN19} recently showed that there exists an unspecified Dirichlet character $\chi$, possibly depending on $f$, such that $N_{f\otimes \chi}^s(T) = \Omega(T^{\frac{1}{6}-\eps})$ for any $\eps > 0$ (see also \cite{CIS13} for a strong result on simple zeros of twists of $f$).  In the same paper, the authors also used the zero-free region of $\Lambda_f$ to slightly limit where pole cancellations in \eqref{delta_twist_simplified_eq} can happen. As a result, they made Booker's result quantitative for $f$ of odd level, showing that $N_f^s(T) = \Omega(\log\log\log{T})$. The restriction $2 \nmid N$ comes from the prominent role played by certain additive twists by $1/2$ in their argument (the use of such twists dates back to the work of Conrey and Ghosh), relying on the fact that there are no non-trivial Dirichlet characters modulo $2$. 

\subsection{Results}

Our main result removes the parity restriction on the level, and rules out complete pole cancellation in \eqref{delta_twist_simplified_eq} on a wide strip, leading to the first power bound for the number of simple zeros of $\Lambda_f$ when $f$ has non-trivial level. 

\begin{theorem}[Power bound for arbitrary level]\label{main_thm}
    Let $f \in S_k(\Gamma_0(N), \xi)$ be a primitive holomorphic modular form of arbitrary weight $k$, level $N$, and nebentypus $\xi$. Then
    \begin{equation*}
        N_f^s(T) = \Omega\left( T^{\delta} \right)
    \end{equation*}
    for any $\delta < \frac{2}{27}$.
\end{theorem}

We obtain a power bound by showing that the aforementioned complete pole cancellation in \eqref{delta_twist_simplified_eq} at a simple zero $\rho$ of $\Lambda_f$ would imply that $\Lambda_{f\otimes\chi}(\rho)=0$ for a large number of characters $\chi$. Such an amount of vanishing can then be ruled out at points $\rho$ close to the line $\Re(s) = 1$, using zero-density results. In order to remove the parity restriction on the level, we get the method started by producing a pole for a certain Dirichlet series via ideas of Booker \cite{Bo16}, instead of relying on the special nature of twists by $1/2$ to do so. See \autoref{sketch_subsection} for a sketch of both arguments.

In \autoref{zero_density_appendix}, we use standard Dirichlet polynomial methods \cite{Mo69a, Mo69b, Hu73, Ju77} to obtain a zero-density bound in degree two which is better in the twist aspect (hence for the application at hand) than other general results from the literature \cite{KM02, LJ19}. It is likely that the exponent in \autoref{main_thm} can be improved by refining this zero-density result, or better yet by dealing directly with non-vanishing at an arbitrary (but fixed) point $\rho$. To be more precise, the problem is to show that the number of primitive characters $\chi \pmod{q}$ with $q \leq Q$ such that $\Lambda_{f\otimes \chi}(\rho) = 0$ is $o\left(Q/\log{Q}\right)$. Using \autoref{zero_density_prop} we obtain this result as long as $\Re(\rho) > \frac{7}{9}$, and the challenge is to enlarge such a half-plane (for instance, the density hypothesis for the family of twists of $f$ would allow one to replace $\frac{7}{9}$ with $\frac{3}{4}$). This type of non-vanishing problem for families has received considerable attention at the central point \cite{So00, KMV00, IS00}, but much less seems to be known in general, and we hope that providing an application will lead to further study. An important feature is that we require more than a $100\%$ rate of non-vanishing, and in fact wish to rule out a thin set of zeros, of size less than the square-root of the size of the family.

Finally, we also improve the exponent in the result of Conrey and Ghosh \cite{CG88} from $1/6$ to $1/5$.

\begin{theorem}[Improved exponent for full level]\label{level_1_thm}
    Let $f \in S_k(\Gamma_0(1))$ be a primitive holomorphic modular form of arbitrary weight $k$ for the full modular group. Then
    \begin{equation*}
        N_f^s(T) = \Omega\left( T^{\nu} \right)
    \end{equation*}
    for any $\nu < \frac{1}{5}$.
\end{theorem}

\autoref{level_1_thm} comes from a simple modification of the last step of their original argument (or its reformulated version in the language of this paper, as presented in \autoref{level_one_section}). Instead of using Weyl subconvexity for $\Lambda_f$, we input Jutila's sixth moment bound \cite{Ju87}. Analogous improvements in the exponent of \autoref{main_thm} would also follow if one had a similar moment bound for $f$ of general level, which may be accessible with current tools but we do not pursue here.

It seems likely that the methods of this paper would apply to Maass forms as well, along the lines of work of Booker, Cho, and Kim \cite{BCK19}. Indeed, while we do use the Ramanujan conjecture for convenience, the argument only really requires information which is already provided by Rankin-Selberg. We restrict ourselves to holomorphic forms for simplicity.

\subsection{Sketch of the argument}\label{sketch_subsection}

Let us describe the obstructions that arise when the level is non-trivial. First we must give an overview of the general method, but we shall be somewhat imprecise and use standard notations that will be familiar to the experts without further explanation, postponing the definitions until \autoref{setup_section}. The fundamental object is the Dirichlet series
\begin{equation}\label{D_def_eq}
    D_f(s) := L_f(s) \left(\frac{L'_f}{L_f}\right)'(s) = \sum_{n=1}^\infty c_f(n) n^{-s} \quad \quad \text{for } \Re(s) > 1,
\end{equation}
which has meromorphic continuation to $\C$ with poles exactly at the simple zeros of $L_f(s)$ (the incomplete $L$-function of $f$), including the trivial ones at $s = \frac{1-k}{2} - n$, for $n \in \Z_{\geq 0}$. It is convenient to work with the completed version $\Delta_f(s) := \Gamma_\C\left( s+\frac{k-1}{2}\right) D_f(s)$, which satisfies a certain functional equation coming from that of $\Lambda_f$. 

The way we obtain information about simple zeros is using the inverse Mellin transform
\begin{equation*}
    F_{f}(z) := 2 \sum_{n=1}^\infty c_{f}(n) n^\frac{k-1}{2} e(nz) = \frac{1}{2\pi i } \int\limits_{\Re(s) = 2} \Delta_{f}(s)(-iz)^{-s - \frac{k-1}{2}} \dd s,
\end{equation*}
for $z \in \H$. Indeed, shifting the line of integration to the left of the critical strip and returning to the right via the functional equation of $\Delta_f$, we pick up poles of $\Delta_f$ and obtain a relation of the form
\begin{equation}\label{simplified_reciprocity_eq}
    F_f(z) = (*) \cdot F_{\overline{f}}\left(-\frac{1}{Nz} \right) + S_f(z) + (**)
\end{equation}
for certain terms $(*)$ and $(**)$ that we brush aside for now. Here the poles contribute
\begin{equation}\label{S_f_simplified_eq}
    S_{f}(z) := -\sum_{\rho} \Lambda'_{f}(\rho) (-iz)^{-\rho - \frac{k-1}{2}},
\end{equation}
where $\rho$ runs over the simple zeros of $\Lambda_f$. 

Understanding the size of $S_f$ gives information about the simple zeros of $\Lambda_f$. To do so we apply a Mellin transform to \eqref{simplified_reciprocity_eq} along the half-line $\Re(z) = \alpha \in \Q^\times$. This gives rise to additive twists of $\Delta_f$, and in the end one obtains a relation between the Mellin transform of $S_f$ and an expression of the form
\begin{equation}\label{delta_twist_simplified_eq}
    \Delta_f(s, \alpha) - (***) \cdot \Delta_{\overline{f}}\left(s, -\frac{1}{N\alpha} \right)
\end{equation}
for some non-vanishing factor $(***)$ which we ignore in this sketch. 

The goal now becomes to show that \eqref{delta_twist_simplified_eq} has a pole with large real part (i.e.\ at least $1/2$) for some $\alpha \in \Q^\times$, since then this pole gets transferred to the Mellin transform of $S_f$ and we get a lower bound for simple zeros. As an aside, the reason why the method produces omega results  is that we obtain only minimal information about the pole structure of the Mellin transform of $S_f$ (which makes the application of Tauberian theorems difficult), as opposed to bounds for $S_f$ itself. 

Since additive twists of $\Delta_f$ are not so well-behaved, we expand them into multiplicative twists instead to understand their poles. At least for $\alpha = \frac{a}{q}$, with $q \nmid N$ a prime, we obtain
\begin{equation}\label{additive_to_mult_simplified_eq}
    \Delta_f\left(s, \frac{a}{q}\right) = \Delta_f(s) +  b_{\chi_0, a} \cdot \Delta_f(s, \chi_0) + \sum_{\substack{\chi \Mod{q} \\ \chi \not= \chi_0}} b_{\chi, a} \cdot \Delta_{f \otimes \chi}(s)
\end{equation}
for certain coefficients $b_{\chi, a}$, where $\chi_0 \pmod{q}$ denotes the trivial character. A key point is that the term $\Delta_f(s) + b_{\chi_0, a} \cdot \Delta_f(s, \chi_0)$ has the same poles as $\Delta_f(s)$ in the interior of the critical strip.

Here it becomes clear why the case $N=1$ is special: one may simply plug $\alpha = \frac{1}{2}$ into \eqref{delta_twist_simplified_eq}. Applying \eqref{additive_to_mult_simplified_eq} and using the fact that there are no non-trivial characters modulo $2$, one checks that \eqref{delta_twist_simplified_eq} has the same poles as $\Delta_f(s)$ inside the critical strip (hence by the aforementioned result of Booker \cite{Bo16} it has at least one pole with real part greater or equal to $1/2$, and one recovers the bound of Conrey and Ghosh). 

For non-trivial level, as was pointed out in \cite{BMN19}, one encounters obstacles that are reminiscent of the difficulties in extending Hecke's converse theorem to arbitrary level. However, Booker, Milinovich, and Ng are still able to obtain a result for $N$ odd, using not only the special nature of the choice $\alpha = \frac{1}{2}$, but also adding an extra additive twist in the outset of the problem and leveraging various choices of $\alpha$ against each other. 

The improvements of the present work are twofold, and in essentially disjoint parts of the argument sketched above. To obtain a result for $f$ of any level (without parity restrictions), instead of using twists by $1/2$ we provide in \autoref{pole_existence_section} a new unified way of verifying that \eqref{delta_twist_simplified_eq} has poles with real part greater or equal to $1/2$ for some choice of $\alpha \in \Q^\times$. The main point is that it is possible to construct a linear combination of certain terms of the form \eqref{delta_twist_simplified_eq} that equals
\begin{equation}\label{linear_comb_simplified_eq}
    \Delta_f\left(s, \frac{1}{p} \right) - \Delta_f\left(s, -\frac{\overline{N}}{p} \right)
\end{equation}
for a certain prime $p$. Then one may use ideas of Booker \cite{Bo16} to show that \eqref{linear_comb_simplified_eq} has a pole inside the critical strip, ultimately coming from the simple zeros of local factors of $\Lambda_f$. We refer to \autoref{pole_existence_overview_subsection} for a more detailed discussion of this part of the argument, and to \cite{Bo03, BK14} and the references therein for striking uses of similar ideas.

To upgrade such a pole inside the critical strip to one with real part greater or equal to $1/2$, we use the important feature that \eqref{linear_comb_simplified_eq} was constructed specifically to satisfy a certain functional equation relating $s$ to $1-s$ (reminiscent of Voronoi summation). Thus the poles of \eqref{linear_comb_simplified_eq} inside the critical strip are invariant under reflection through the central point, which gives the desired pole with real part at least $1/2$ and makes the method applicable to all $N$.

We now turn to the second improvement, which is what allows us to obtain a power bound. Observe from \eqref{S_f_simplified_eq} that the contribution to $S_f$ of each simple zero $\rho$ is weighted by a factor that becomes larger with $\Re(\rho)$, so in its current form the result is poorer if $\Lambda_f(s)$ has simple zeros close to $\Re(s) = 1$. If all the simple zeros $\rho$ satisfy $\Re(\rho) \leq \frac{7}{9}$, then we simply use the argument above and obtain a power bound for the number of simple zeros of $\Lambda_f$. Otherwise, if $\rho$ is a simple zero with $\Re(\rho) > \frac{7}{9}$, we will show that there exists $\alpha \in \Q^\times$ such that \eqref{delta_twist_simplified_eq} also has a pole at $\rho$ (in \cite{BMN19} the key to control this scenario is using the zero-free region of $\Lambda_f$ to limit $\Re(\rho)$, which is why the resulting bound is of logarithmic quality). The pole of \eqref{delta_twist_simplified_eq} at $\rho$ ultimately gives an even better power bound, so either way we obtain the desired result.

Let $\rho$ be a simple zero of $\Lambda_f$ (therefore a pole of $\Delta_f$) with $\Re(\rho) > \frac{7}{9}$. To rule out pole cancellations in \eqref{delta_twist_simplified_eq} for every $\alpha \in \Q^\times$, we introduce a new number-theoretic input into the argument, namely a zero-density estimate for twists of $f$. This is done by observing that for any prime $p \equiv 1 \pmod{N}$ there is a linear combination of terms of the form \eqref{delta_twist_simplified_eq} that gives
\begin{equation}\label{power_bound_linear_comb_eq}
    p^{1-2s} \Delta_f(s) - \Delta_f\left(s, \frac{1}{p}\right).
\end{equation}
One can use \eqref{additive_to_mult_simplified_eq} to understand \eqref{power_bound_linear_comb_eq}, concluding that it is equal (modulo a term that is holomorphic at $s=\rho$) to
\begin{equation}\label{mult_expansion_simplified_eq}
    b_{f, \chi_0}(s) \cdot \Delta_f(s) + \sum_{\substack{\chi \Mod{p} \\ \chi \not= \chi_0}} b_{\chi} \cdot \Delta_{f \otimes \chi}(s)
\end{equation}
for some factor $b_{f, \chi_0}(s)$ which is non-vanishing inside the critical strip (hence at $s=\rho$).

If \eqref{power_bound_linear_comb_eq} does not have a pole at $s=\rho$, then the pole of $\Delta_f(s)$ there must be cancelled in \eqref{mult_expansion_simplified_eq}, so $\Delta_{f\otimes\chi}(s)$ must have a pole at $s=\rho$ for at least one non-trivial character $\chi \pmod{p}$. This implies that $\Lambda_{f\otimes\chi}(\rho) = 0$ for at least one non-trivial character $\chi$ modulo every prime $p \equiv 1\pmod{N}$. However, since $\Re(\rho) > \frac{7}{9}$, we can rule this out via zero-density estimates for twists of $f$. Therefore we show that \eqref{delta_twist_simplified_eq} has a pole at $s=\rho$ for some $\alpha \in \Q^\times$, which implies a power bound for the number of simple zeros of $\Lambda_f$.

\subsection*{Acknowledgments}

I would like to thank my PhD advisor, Maksym Radziwi{\l}{\l}, for introducing me to this problem, for many useful discussions on the topics of this paper, and for valuable advice and encouragement. Thanks also to Andrew Booker for his helpful comments and correspondence.


\section{The setup}\label{setup_section}

\subsection{Definitions and background}

Let $f \in S_k(\Gamma_0(N), \xi)$ be a primitive form (i.e.\ an arithmetically normalized holomorphic Hecke cusp newform) of arbitrary weight $k$, level $N$, and nebentypus character $\xi \pmod{N}$. Writing the Fourier expansion
\begin{equation*}
    f(z) = \sum_{n=1}^\infty \lambda_f(n) n^\frac{k-1}{2} e(nz)
\end{equation*}
for $z \in \H$, where $\lambda_f(1)=1$, we have Deligne's bound $|\lambda_f(n)| \leq d(n)$. Associate to $f$ the usual completed $L$-function $\Lambda_f(s) := \Gamma_\C\left(s + \frac{k-1}{2}\right) L_f(s)$, which is entire, where $\Gamma_\C\left(s\right) := 2 (2\pi)^{-s}\Gamma(s)$ and
\begin{equation*}
     L_f(s) := \sum_{n=1}^\infty \lambda_f(n) n^{-s} = \prod_{p \text{ prime}} \left(1 - \lambda_f(p) p^{-s} + \xi(p)p^{-2s} \right)^{-1} \quad \quad \text{for } \Re(s) > 1.
\end{equation*}
Then we have the functional equation $\Lambda_f(s) = \rn_f N^{\frac{1}{2}-s} \Lambda_{\overline{f}}(1-s)$, where $\overline{f} \in S_k(\Gamma_0(N), \overline{\xi})$ is the dual of $f$, with Fourier coefficients $\lambda_{\overline{f}}(n) = \overline{\lambda_f(n)}$, and $\rn_f \in \C$ is the root number of $f$, with $|\rn_f| = 1$.

Let $D_f$ be as in \eqref{D_def_eq}. For $\alpha \in \Q^\times$, $\chi$ a Dirichlet character, and $\Re(s) > 1$, we define the additive twists
\begin{equation*}
    L_f(s, \alpha) := \sum_{n=1}^\infty \lambda_f(n) e(n\alpha) n^{-s} \quad \quad \text{and} \quad \quad D_f(s, \alpha) := \sum_{n=1}^\infty c_f(n) e(n\alpha) n^{-s},
\end{equation*}
and the multiplicative twists
\begin{equation*}
    L_f(s, \chi) := \sum_{n=1}^\infty \lambda_f(n) \chi(n) n^{-s} \quad \quad \text{and} \quad \quad D_f(s, \chi) := \sum_{n=1}^\infty c_f(n) \chi(n) n^{-s}.
\end{equation*}

Denote
\begin{equation*}
    Q(N) := \{1\} \cup \{p \text{ prime}: p \nmid N \}.
\end{equation*}
For each Dirichlet character $\chi \pmod{q}$, there is a unique primitive form $f \otimes \chi$ such that $\lambda_{f \otimes \chi}(n) = \lambda_f(n) \chi(n)$ for every $n$ with $(n, q) = 1$, by \cite[Theorem 3.2]{AL78}. If $q \in Q(N)$ and $\chi$ is non-trivial, then in fact $L_{f}(s, \chi) = L_{f\otimes \chi}(s)$ and therefore this multiplicative twist has analytic continuation to $\C$. This shows that $D_f(s, \chi) = L_f(s, \chi) \left(\frac{L'_f(s, \chi)}{L_f(s, \chi)} \right)' = D_{f\otimes \chi}(s)$ has meromorphic continuation to $\C$. 

Similar results hold for the additive twists as well. Indeed, if $q \in Q(N)$, then we can expand our additive characters into multiplicative ones using
\begin{equation}\label{additive_char_expansion_eq}
    e\left(\frac{n}{q}\right) = \frac{q-1}{\phi(q)} + \frac{q}{\phi(q)}\tau(\chi_0)\chi_0(n) + \frac{1}{\phi(q)}\sum_{\substack{\chi \Mod{q} \\ \chi \not= \chi_0}} \tau(\overline{\chi}) \chi(n),
\end{equation}
where $\chi_0 \pmod{q}$ is the trivial character, the sum ranges over every non-trivial $\chi \pmod{q}$, and $\tau$ denotes the Gauss sum (observe that $\tau(\chi_0) = 1$ if $q=1$ and $\tau(\chi_0) = -1$ otherwise). For any $a \in \Z$, this implies that $L_f\left(s, \frac{a}{q}\right)$ is entire, and $D_f\left(s, \frac{a}{q}\right)$ extends meromorphically to $\C$. 

To be more precise, for $q \in Q(N)$, consider the rational functions
\begin{equation*}
    P_{f, q}(x) := 
    \begin{cases}
        1 & \text{if } q=1, \\
        1 -\lambda_f(q)x + \xi(q) x^2 & \text{otherwise},
    \end{cases}
\end{equation*}
and
\begin{equation*}
    R_{f, q}(x) :=
    \begin{cases}
        0 & \text{if } q=1, \\
        \frac{q \log^2{q}}{\phi(q)} \frac{x (\lambda_f(q) - 4 \xi(q)x + \lambda_f(q) \xi(q) x^2)}{P_{f, q}(x)} & \text{otherwise}.
    \end{cases}
\end{equation*}
Then \eqref{additive_char_expansion_eq} gives
\begin{equation}\label{D_expansion_eq}
    D_f\left(s, \frac{a}{q}\right) = \frac{q-1}{\phi(q)}D_f(s) + \frac{q}{\phi(q)}\tau(\chi_0)\chi_0(a) D_f(s, \chi_0) + \frac{1}{\phi(q)}\sum_{\substack{\chi \Mod{q} \\ \chi \not= \chi_0}} \tau(\overline{\chi}) \chi(a) D_f(s, \chi).
\end{equation}
We have seen before that if $\chi \pmod{q}$ is non-trivial then $D_f(s, \chi) = D_{f\otimes \chi}(s)$ extends meromorphically to $\C$, but also from $D_f(s, \chi_0) = L_f(s, \chi_0) \left(\frac{L'_f(s, \chi_0)}{L_f(s, \chi_0)}\right)'$ and $L_f(s, \chi_0) = P_{f, q}(q^{-s}) L_f(s)$ (coming from the Euler product of $L_f$) we get
\begin{equation}\label{D_trivial_char_eq}
    D_f(s, \chi_0) = P_{f, q}(q^{-s}) D_f(s) - \frac{\phi(q)}{q} R_{f, q}(q^{-s}) L_f(s),
\end{equation}
and this provides the meromorphic continuation of $D_f\left(s, \frac{a}{q}\right)$ to $\C$. The analytic continuation of $L_f\left(s, \frac{a}{q}\right)$ to $\C$ follows in the same way.

For $(a, q) = 1$, it will be convenient to work with
\begin{equation*}
    D_{f, a, q}(s) := D_f\left(s, \frac{a}{q}\right) - R_{f, q}(q^{-s}) L_f(s) = \sum_{n=1}^\infty c_{f, a, q}(n) n^{-s} \quad \quad \text{for } \Re(s) > 1,
\end{equation*}
where the Dirichlet series expansion follows from \eqref{D_expansion_eq}, \eqref{D_trivial_char_eq}, and \eqref{D_def_eq}. Clearly $D_{f, a, q}(s)$ extends meromorphically to $\C$. We then define additive and multiplicative twists of $D_{f, a, q}(s)$. Namely, if $\chi$ is a Dirichlet character and $\alpha \in \Q^\times$, then for $\Re(s) > 1$ we let
\begin{equation*}
    D_{f, a, q}(s, \chi) := \sum_{n=1}^\infty c_{f, a, q}(n) \chi(n) n^{-s} \quad \quad \text{and} \quad \quad D_{f, a, q}(s, \alpha) := \sum_{n=1}^\infty c_{f, a, q}(n) e(n\alpha) n^{-s}.
\end{equation*}

Finally, associate to each of $L_f, D_f, D_{f, a, q}$ and their (additive or multiplicative) twists the completed versions $\Lambda_f, \Delta_f, \Delta_{f, a, q}$, respectively, obtained by multiplying by $\Gamma_\C\left(s + \frac{k-1}{2}\right)$. 

\subsection{Functional equations}

If $q \in Q(N)$ and $\chi \pmod{q}$ is non-trivial, the functional equation for $f \otimes \chi$ gives
\begin{equation*}
    \Lambda_f(s, \chi) = \rn_f \xi(q) \chi(N) \frac{\tau(\chi)^2}{q} (Nq^2)^{\frac{1}{2}-s} \Lambda_{\overline{f}}(1-s, \overline{\chi}),
\end{equation*}
and as a consequence we obtain the corresponding functional equation for $\Delta_f(s, \chi) = \Delta_{f\otimes\chi}(s)$, given by
\begin{equation*}
    \Delta_f(s, \chi) - \rn_f \xi(q) \chi(N) \frac{\tau(\chi)^2}{q} (Nq^2)^{\frac{1}{2}-s} \Delta_{\overline{f}}(1-s, \overline{\chi}) = \Lambda_f(s, \chi) \left(\psi'\left(\frac{k+1}{2}-s\right) - \psi'\left(s + \frac{k-1}{2}\right)\right),
\end{equation*}
where $\psi(s) := \frac{\Gamma'}{\Gamma}(s)$. Combining that with the relation
\begin{equation}\label{Delta_additive_twist_decomposition_eq}
    \Delta_{f, a, q}(s) = \left(\frac{q-1}{\phi(q)} + \frac{q}{\phi(q)}\tau(\chi_0) P_{f, q}(q^{-s}) \right) \Delta_f(s) + \frac{1}{\phi(q)}\sum_{\substack{\chi \Mod{q} \\ \chi \not= \chi_0}} \tau(\overline{\chi}) \chi(a) \Delta_f(s, \chi)
\end{equation}
for $q \in Q(N)$ and $(a, q) = 1$, which follows from \eqref{D_expansion_eq} and \eqref{D_trivial_char_eq}, we obtain a functional equation for additive twists of $\Delta_f$.

\begin{proposition}[Functional equation for $\Delta_{f, a, q}$ {\cite[Proposition 2.1]{BMN19}}]\label{delta_twist_funct_eq_prop}
    Let $f \in S_k(\Gamma_0(N), \xi)$ be a primitive form, $q \in Q(N)$, and $a \in \Z$ coprime to $q$. Then
    \begin{equation*}
        \Delta_{f, a, q}(s) - \rn_f \xi(q) (N q^2)^{\frac{1}{2}-s} \Delta_{\overline{f}, -\overline{Na}, q}(1-s) = \Lambda_f\left(s, \frac{a}{q} \right) \left(\psi'\left(\frac{k+1}{2}-s\right) - \psi'\left(s + \frac{k-1}{2}\right)\right).
    \end{equation*}
\end{proposition}

\subsection{The detection mechanism for simple zeros}

We give a brief account of the techniques of \cite{BMN19}, since they will be relevant in what follows, but refer to that paper for details. The main idea originates in \cite{CG88}, and is developed in greater generality in \cite{Bo16}. The starting point is to study the poles of $\Delta_f$ by relating them to the inverse Mellin transform of $\Delta_f$, via a contour integral. We develop the notation in the more general case of $\Delta_{f, a, q}$ for future reference. For $z \in \H$, let
\begin{equation*}
    \begin{split}
        &F_{f, a, q}(z) := 2 \sum_{n=1}^\infty c_{f, a, q}(n) n^\frac{k-1}{2} e(nz), \quad \quad S_{f, a, q}(z) := \sum_{\Re(\rho) \in (0, 1)} \Res_{s=\rho} \Delta_{f, a, q}(s) (-iz)^{-s - \frac{k-1}{2}}, \\
        &A_{f, a, q}(z) := \frac{1}{2\pi i} \int\limits_{\Re(s) = \frac{k}{2}} \Lambda_f \left(s, \frac{a}{q}\right) \left(\psi'\left( s + \frac{k-1}{2}\right) + \psi'\left(s - \frac{k-1}{2}\right)\right) (-iz)^{-s -\frac{k-1}{2}} \dd s,
    \end{split}
\end{equation*}
and
\begin{equation*}
    B_{f, a, q}(z) := \frac{1}{2\pi i} \int\limits_{\Re(s) = \frac{k}{2}} \Lambda_f \left(s, \frac{a}{q}\right) \frac{\pi^2}{\sin^2\left(\pi\left(s + \frac{k-1}{2}\right)\right)} (-iz)^{-s -\frac{k-1}{2}} \dd s,
\end{equation*}
where $(-iz)^{-s-\frac{k-1}{2}}$ is defined in terms of the principal branch of $\log(-iz)$.

Taking the inverse Mellin transform of $\Delta_{f, a, q}$ (evaluated at $-iz$), shifting the line of integration to the left of the critical strip --- where we pick up the factor $S_{f, a, q}$ corresponding to the poles --- and using the functional equation (\autoref{delta_twist_funct_eq_prop}) to return to the right of the critical strip, we obtain (see \cite[Lemma 2.3]{BMN19} for details) the relation
\begin{equation*}
    S_{f, a, q}(z) = F_{f, a, q}(z) - \frac{\rn_f \xi(q)}{(-i \sqrt{N}qz)^k} F_{\overline{f}, -\overline{Na}, q}\left(-\frac{1}{Nq^2z}\right) + A_{f, a, q}(z) - B_{f,a , q}(z).
\end{equation*}

The next step is to take the Mellin transform for $z \in \H$ along a vertical line in the relation above. Such a procedure along the line $\Re(z) = 0$ would essentially bring us back to the previous step, but we instead integrate along $\Re(z) = \alpha$ for some $\alpha \in \Q^\times$ and obtain additive twists. The final result is the following.

\begin{proposition}[Detecting poles of $\Delta_{f, a, q}$ via further additive twists {\cite[Proposition 2.2]{BMN19}}]\label{pole_detection_prop}
    Define 
    \begin{equation*}
        H_{f, a, q, \alpha}(s) := \Delta_{f, a, q}(s, \alpha) - \rn_f \xi(q) (i \sgn(\alpha))^k (N q^2 \alpha^2)^{s - \frac{1}{2}} \Delta_{\overline{f}, -\overline{Na}, q}\left(s, - \frac{1}{Nq^2\alpha}\right)
    \end{equation*}
    and
    \begin{equation*}
        I_{f, a, q, \alpha}(s) := \int_0^{\frac{|\alpha|}{4}} S_{f, a, q}(\alpha + iy) y^{s + \frac{k-1}{2}} \frac{dy}{y}.
    \end{equation*}
    Then $I_{f, a, q, \alpha}(s) - H_{f, a, q, \alpha}(s)$ has analytic continuation to $\Re(s) > 0$. Therefore, if
    \begin{equation}\label{S_abs_val_condition_eq}
        \int_0^{\frac{|\alpha|}{4}} |S_{f, a, q}(\alpha + iy)| y^{\sigma + \frac{k-1}{2}} \frac{dy}{y} < \infty
    \end{equation}
    for some $\sigma \geq 0$, then $H_{f, a, q, \alpha}(s)$ is holomorphic for $\Re(s) > \sigma$.
\end{proposition}

We will use only the special case $a=q=1$ (i.e.\ detecting poles of $\Delta_f$, or equivalently simple zeros of $L_f$) of \autoref{pole_detection_prop}, but the method of proof used for the general case $\Delta_{f, a, q}$ will be the key for showing that $H_{f, 1, 1, \alpha}$ has a pole in the critical strip, for some $\alpha \in \Q^\times$. For convenience, from now on we denote $H_{f, \alpha} := H_{f, 1, 1, \alpha}$.


\section{Existence of poles of \texorpdfstring{$H_{f, \alpha}$}{H {f, alpha}}}\label{pole_existence_section}

\subsection{Outline of the method}\label{pole_existence_overview_subsection}

To establish an abundance of simple zeros of $L_f$ (i.e.\ poles of $\Delta_f$), we will use the poles of $H_{f, \alpha}$ in the critical strip, since through \eqref{S_abs_val_condition_eq} their existence would imply that $S_{f, 1, 1}$ cannot be always small. However, showing that even a single such pole of $H_{f, \alpha}$ exists turns out to be difficult, since one needs to rule out a cancellation of poles between the two terms of $H_{f, \alpha}$. The purpose of this section is to establish such a result.

In \cite{BMN19} the authors circumvent this issue in the case $2 \nmid N$ by exploiting the relations between the $H_{f, a, q, \alpha}$ for various choices of parameters $(a, q, \alpha)$. The limitation on the level $N$ comes from the key role played by twists by $1/2$ (which also play an important role in \cite{CG88}), since the poles of $\Delta_{f, 1, 2}$ are easily understood in terms of those of $\Delta_f$, due to \eqref{Delta_additive_twist_decomposition_eq} and the fact that there are no non-trivial characters modulo $2$. The issue is that this line of argument requires the case $q=2$ of \autoref{pole_detection_prop}, which is not available if $2 \mid  N$ since the functional equation in \autoref{delta_twist_funct_eq_prop} no longer holds, as the local factor for a prime dividing the level has different, more problematic properties. 

We will follow a different approach based on the methods of \cite{Bo16}, where a significant difficulty is showing that $\Delta_f$ has even a single pole in the critical strip, and this is reminiscent of our situation for $H_{f, \alpha}$. The argument in the reference proceeds by contradiction, using a refined version of \autoref{pole_detection_prop} with lower order terms to obtain a relation between $\Delta_f(s, \alpha)$ and $\Delta_f(s+1, \alpha_i)$ for certain $\alpha_i$, in the absence of poles. This would give a holomorphic continuation of $\Delta_f(s, \alpha)$ to the line $\Re(s) = 0$, where for certain $\alpha$ it plainly has poles coming from simple zeros of local factors of $L_f$ at primes dividing the denominator of $\alpha$. We apply this argument for a certain difference of $L$-functions related to $H_{f, \alpha}$, instead of for $\Delta_f$, and our issue of ruling out cancellations of poles in $H_{f, \alpha}$ at unknown locations inside the critical strip reduces to the simpler task of ruling out such cancellations at the simple zeros of certain local factors, where this can be explicitly done.

\subsection{Implementation}

Observe that
\begin{equation*}
    H_{f, 1}(s) = \Delta_f(s) - \rn_f i^k N^{s - \frac{1}{2}} \Delta_{\overline{f}}\left(s, -\frac{1}{N}\right),
\end{equation*}
and if $p$ is a prime satisfying $p \equiv 1 \pmod{N}$ then $\Delta_{\overline{f}}\left(s, -\frac{p}{N} \right) = \Delta_{\overline{f}}\left(s, -\frac{1}{N} \right)$, so
\begin{equation*}
    H_{f, \frac{1}{p}}(s) = \Delta_f \left(s, \frac{1}{p}\right) - \rn_f i^k \left(\frac{N}{p^2}\right)^{s- \frac{1}{2}} \Delta_{\overline{f}}\left(s, -\frac{1}{N}\right).
\end{equation*}
Therefore, 
\begin{equation}\label{first_H_diff_eq}
    \begin{split}
        p^{1-2s} H_{f, 1}(s) - H_{f, \frac{1}{p}}(s) &= p^{1-2s}\Delta_f(s) - \Delta_f\left(s, \frac{1}{p}\right) \\
        &= p^{1-2s}\Delta_f(s) - \Delta_{f, 1, p}(s) + R_{f, p}(p^{-s}) \Lambda_f(s).
    \end{split}
\end{equation}

Similarly, if we let $d := \frac{p-1}{N} \in \Z_{>0}$ then
\begin{equation*}
    H_{f, d}(s) = \Delta_f(s) - \rn_f i^k (Nd^2)^{s-\frac{1}{2}} \Delta_{\overline{f}}\left(s, -\frac{1}{Nd}\right),
\end{equation*}
and $\Delta_{\overline{f}}\left(s, -\frac{p}{Nd}\right) = \Delta_{\overline{f}}\left(s, -\frac{1}{Nd}\right)$, so
\begin{equation*}
    H_{f, \frac{d}{p}}(s) = \Delta_f\left(s, \frac{d}{p}\right) - \rn_f i^k \left(\frac{N d^2}{p^2}\right)^{s-\frac{1}{2}} \Delta_{\overline{f}}\left(s, -\frac{1}{Nd}\right).
\end{equation*}
Therefore, since $d \equiv -\overline{N} \pmod{p}$,  
\begin{equation}\label{second_H_diff_eq}
    \begin{split}
        p^{1-2s} H_{f, d}(s) - H_{f, \frac{d}{p}}(s) &= p^{1-2s}\Delta_f(s) - \Delta_f\left(s, \frac{d}{p}\right) \\
        &= p^{1-2s}\Delta_f(s) - \Delta_{f, -\overline{N}, p}(s) + R_{f, p}(p^{-s}) \Lambda_f(s).
    \end{split}
\end{equation}

Subtracting \eqref{first_H_diff_eq} from \eqref{second_H_diff_eq}, we conclude that
\begin{equation}\label{key_H_diff_eq}
    p^{1-2s} H_{f, d}(s) - H_{f, \frac{d}{p}}(s) - p^{1-2s} H_{f, 1}(s) + H_{f, \frac{1}{p}}(s)= \Delta_{f, 1, p}(s) - \Delta_{f, -\overline{N}, p}(s).
\end{equation}
We will be able to show the existence of useful poles for at least one of $H_{f, 1}(s), H_{f, \frac{1}{p}}(s), H_{f, d}(s)$, or $H_{f, \frac{d}{p}}(s)$ using the key proposition below.

\begin{proposition}[Ruling out complete cancellation of poles]\label{pole_existence_prop}
    For any prime $p \not= N+1$ such that $p \equiv 1 \pmod{N}$, the meromorphic function
    \begin{equation*}
        G_{f, p}(s) := \Delta_{f, 1, p}(s) - \Delta_{f, -\overline{N}, p}(s)
    \end{equation*}
    has at least one pole in $\Re(s) \in (0, 1)$.
\end{proposition}

\begin{remark}
    Our proof of \autoref{pole_existence_prop} can easily be adapted to obtain infinitely many poles of $G_{f, p}(s)$ in $\Re(s) \in (0, 1)$. Such a result has the same strength for our application as the existence of a single pole, so for simplicity we stick with the current statement.
\end{remark}

Assuming \autoref{pole_existence_prop}, we have the following consequence which will be the starting point in the course of our subsequent analysis.

\begin{proposition}[Existence of poles with large real part]\label{H_large_pole_existence_prop}
    There exists $\alpha_f \in \Q^\times$ such that at least one of $H_{f, \alpha_f}(s)$ or $H_{\overline{f}, \alpha_f}(s)$ has a pole in $\Re(s) \in \left[\frac{1}{2}, 1\right)$.
\end{proposition}

\begin{proof}
    For any prime $p \not= N+1$ such that $p \equiv 1 \pmod{N}$, from the functional equation in \autoref{delta_twist_funct_eq_prop} we have that
    \begin{equation*}
        G_{f, p}(s) + \rn_f (Np^2)^{\frac{1}{2}-s} G_{\overline{f}, p}(1-s) = \left(\Lambda_f\left(s, \frac{1}{p}\right) - \Lambda_f\left(s, -\frac{\overline{N}}{p}\right)\right) \left( \psi' \left(\frac{k+1}{2}-s\right) - \psi'\left(s + \frac{k-1}{2}\right) \right),
    \end{equation*}
    as $\xi(p) = 1$. Since $\Lambda_f\left(s, \frac{1}{p}\right)$ and $\Lambda_f\left(s, -\frac{\overline{N}}{p}\right)$ are both entire, as easily follows from expanding into characters (see \eqref{lambda_additive_twist_char_expansion_eq} below for details), and the poles of $\psi'(s)$ coincide with the poles of $\Gamma(s)$, we conclude that $G_{f, p}(s)$ and $G_{\overline{f}, p}(1-s)$ have the same poles in $\Re(s) \in (0, 1)$, as the RHS of the equation above is holomorphic in that region.
    
    Combining this with \autoref{pole_existence_prop}, we get that at least one of $G_{f, p}(s)$ or $G_{\overline{f}, p}(s)$ has a pole in $\Re(s) \in \left[\frac{1}{2}, 1\right)$, so \eqref{key_H_diff_eq} shows that the desired result holds for some $\alpha_f \in \left\{d, \frac{d}{p}, 1, \frac{1}{p}\right\}$, where $d = \frac{p-1}{N}$ as before.
    
\end{proof}

\subsection{Preliminary results}

Before proceeding to the proof of \autoref{pole_existence_prop}, we take note of certain computations essentially contained in \cite{BMN19} that will be relevant for our argument. Those are reproduced in the auxiliary results below for ease of reference.

\begin{lemma}[Inverse Mellin transform computations]\label{mellin_computation_lemma}
    Let $0 < \eta < 1/2$. Then for $z \in \H$ we have
    \begin{equation*}
        I^R_{f, a, q}(z) := \frac{1}{2\pi i } \int\limits_{\Re(s) = 1 + \eta} \Delta_{f, a, q}(s)(-iz)^{-s - \frac{k-1}{2}} \dd s = F_{f, a, q}(z)
    \end{equation*}
    and
    \begin{equation*}
        \begin{split}
            I^L_{f, a, q}(z) :=& \ \frac{1}{2\pi i } \int\limits_{\Re(s) = -\eta} \Delta_{f, a, q}(s)(-iz)^{-s - \frac{k-1}{2}} \dd s \\
            =& \ \frac{\rn_f \xi(q)}{(-i \sqrt{N}qz)^k} F_{\overline{f}, -\overline{Na}, q}\left(-\frac{1}{Nq^2z}\right) - A_{f, a, q}(z) + B_{f,a , q}(z) - \Res_{s=0} \Delta_{f, a, q}(s).
        \end{split}
    \end{equation*}
\end{lemma}

\begin{proof}
    This follows from the functional equation in \autoref{delta_twist_funct_eq_prop} (for the case of $I^L_{f, a, q}(z)$) and a computation of inverse Mellin transforms. The details are contained in the proof of \cite[Lemma 2.3]{BMN19} --- see in particular equations $(2.9)$ and $(2.12)$ there, and keep in mind that the residue at $s=0$ only contributes if $k=1$. Our statement above corrects a small typo in the computation of this residue at the last display of page 382 of the reference, where the term $\Delta_{f, a, q}\left(s, \frac{a}{q}\right)$ should be replaced by $\Delta_{f, a, q}(s)$, according to the functional equation.
    
\end{proof}

\begin{lemma}[Auxiliary analytic continuations]\label{aux_comp_continuation_lemma}
    Let $\alpha \in \Q^\times$. Then for any $M \in \Z_{\geq 0}$,
    \begin{equation*}
        \begin{split}
            &\int_0^{\frac{|\alpha|}{4}} (-i \sqrt{N} q (\alpha+iy))^{-k} F_{\overline{f}, -\overline{Na}, q}\left(-\frac{1}{Nq^2(\alpha+iy)}\right) y^{s+\frac{k-1}{2}}\frac{dy}{y} \\
            &- (i \sgn(\alpha))^k \sum_{m=0}^{M-1} (-i\alpha)^{-m} \binom{s+m-\frac{k+1}{2}}{m} (Nq^2\alpha^2)^{s-\frac{1}{2}+m} \Delta_{\overline{f}, -\overline{Na}, q}\left(s+m, -\frac{1}{Nq^2\alpha}\right)
        \end{split}
    \end{equation*}
    continues to a holomorphic function in $\{ s \in \C : \Re(s) > 1-M\}$. Furthermore, each of
        \begin{equation*}
        \int_0^{\frac{|\alpha|}{4}} F_{f, a, q}(\alpha + iy) y^{s + \frac{k-1}{2}}\frac{dy}{y} - \Delta_{f, a, q}(s, \alpha),
    \end{equation*}
    \begin{equation*}
        \Gamma_\C(s)^{-1} \int_0^{\frac{|\alpha|}{4}} A_{f, a, q}(\alpha + iy) y^s \frac{dy}{y},
    \end{equation*}
    and
    \begin{equation*}
        \Gamma_\C(s)^{-1} \int_0^{\frac{|\alpha|}{4}} B_{f, a, q}(\alpha + iy) y^s \frac{dy}{y}
    \end{equation*}
    continues to an entire function of $s$.
\end{lemma}

\begin{proof} 
    Those are precisely \cite[Lemmas 2.4, 2.5, 2.6, and 2.7]{BMN19} in our notation. The first result is the hardest to establish, and it follows from Taylor expanding the phases in $F_{\overline{f}, -\overline{Na}, q}$ and carefully analyzing the ensuing Mellin transforms. The idea is that as $z :=\alpha+iy\in \H$ ranges over the vertical half-line $\Re(z) = \alpha$, $w := -\frac{1}{Nq^2z}\in \H$ ranges over a semicircle centered in the $x$-axis with an endpoint at $-\frac{1}{Nq^2 \alpha}$, so to a first approximation the input $w$ of $F_{\overline{f}, -\overline{Na}, q}$ in the first integral can be considered to range over the vertical half-line $\Re(w) = -\frac{1}{Nq^2\alpha}$, which by Mellin inversion gives rise to a term of the form $\Delta_{\overline{f}, -\overline{Na}, q}\left(s, -\frac{1}{Nq^2\alpha} \right)$. The other terms arise from lower order components of the aforementioned Taylor expansion.
    
\end{proof}

\begin{lemma}[Analytic continuation of Mellin transforms]\label{analytic_continuation_lemma}
    Let $\alpha \in \Q^\times$ and $M \in \Z_{\geq 0}$. Then 
    \begin{equation*}
        \int_0^{\frac{|\alpha|}{4}} I^R_{f, a, q}(\alpha + iy) y^{s+\frac{k-1}{2}} \frac{dy}{y} - \Delta_{f, a, q}(s, \alpha)
    \end{equation*}
    continues to an entire function of $s$, and
    \begin{equation*}
        \begin{split}
            &\int_0^{\frac{|\alpha|}{4}} \left(I^L_{f, a, q}(\alpha + iy) + \Res_{s=0} \Delta_{f, a, q}(s)\right) y^{s+\frac{k-1}{2}} \frac{dy}{y} \\
            &- \rn_f \xi(q) (i \sgn(\alpha))^k \sum_{m=0}^{M-1} (-i\alpha)^{-m} \binom{s+m-\frac{k+1}{2}}{m} (Nq^2\alpha^2)^{s-\frac{1}{2}+m} \Delta_{\overline{f}, -\overline{Na}, q}\left(s+m, -\frac{1}{Nq^2\alpha}\right)
        \end{split}
    \end{equation*}
    continues to a meromorphic function in $\{s \in \C : \Re(s) > 1-M\}$ whose only possible poles in that region must be at $s = \frac{1-k}{2} - n$, for $n \in \Z_{\geq 0}$.
\end{lemma}

\begin{proof}
    Follows directly from plugging the equations in \autoref{mellin_computation_lemma} into the integrals above and using \autoref{aux_comp_continuation_lemma} for each term that arises. The only possible poles come from the integral terms corresponding to $A_{f, a, q}$ and $B_{f, a, q}$, whose poles must be poles of $\Gamma_\C\left(s+\frac{k-1}{2}\right)$.
    
\end{proof}

The next two results determine the locations of the poles of $\Delta_{f, a, q}(s)$ and some of its additive twists. \autoref{Delta_additive_twist_pole_location_lemma} is essentially contained in \cite[Proposition 2.2]{BMN19}, while \autoref{Delta_hardcore_pole_location_lemma} requires a more careful analysis.

\begin{lemma}[No exotic poles for $\Delta_{f, a, q}$]\label{Delta_additive_twist_pole_location_lemma}
    The poles of $\Delta_{f, a, q}(s)$ satisfy $\Re(s) \in (0, 1)$ or $s = \frac{1-k}{2}-n$ for some $n\in \Z_{\geq 0}$.
\end{lemma}

\begin{proof}
    First observe that $\Delta_{f, a, q}(s)$ has no poles with $\Re(s) \geq 1$. Indeed, this follows from \eqref{Delta_additive_twist_decomposition_eq} and the fact that for non-trivial $\chi \pmod{q}$ the poles of $\Delta_f(s)$ and $\Delta_{f}(s, \chi) = \Delta_{f\otimes \chi}(s)$ are at simple zeros of $L_f(s)$ and $L_{f\otimes \chi}(s)$, respectively, but there are no such zeros with $\Re(s) \geq 1$ by non-vanishing for automorphic $L$-functions \cite{JS77}. As a consequence, we can also determine the poles of $\Delta_{f, a, q}(s)$ with $\Re(s) \leq 0$, through the functional equation. Using \eqref{additive_char_expansion_eq} and $\Lambda_f(s, \chi_0) = P_{f, q}(q^{-s}) \Lambda_f(s)$, since $(a, q)=1$ we get
    \begin{equation}\label{lambda_additive_twist_char_expansion_eq}
        \Lambda_f\left(s, \frac{a}{q}\right) = \left(\frac{q-1}{\phi(q)} + \frac{q}{\phi(q)} \tau(\chi_0) P_{f, q}(q^{-s}) \right) \Lambda_f(s) + \frac{1}{\phi(q)} \sum_{\substack{\chi \Mod{q} \\ \chi \not= \chi_0}} \tau(\overline{\chi}) \chi(a) \Lambda_{f \otimes \chi}(s),
    \end{equation}
    so $\Lambda_f\left(s, \frac{a}{q}\right)$ is entire. The poles of $\psi'(s)$ coincide with the poles of $\Gamma(s)$, so \autoref{delta_twist_funct_eq_prop} shows that $\Delta_{f, a, q}(s)$ has no poles with $\Re(s)\leq 0$, except possibly for $s = \frac{1-k}{2}-n$, for some $n\in \Z_{\geq 0}$.
    
\end{proof}

\begin{lemma}[Location of exotic poles for additive twists of $\Delta_{f, a, p}$]\label{Delta_hardcore_pole_location_lemma}
    Let $p, q \in Q(N)$ with $p \not= q$, and let $a, b \in \Z$ with $(a, p)= (b, q) = 1$. If we let $\chi_0 \pmod{q}$ and $\psi_0 \pmod{p}$ denote the trivial characters, then
    \begin{equation*}
        \begin{split}
            \Delta_{f, a, p}\left(s, \frac{b}{q}\right) &+ \tau(\chi_0) \left(\frac{p-1}{\phi(p)} + \frac{p}{\phi(p)} \tau(\psi_0) P_{f, p}(p^{-s})\right) R_{f, q}(q^{-s}) \Lambda_f(s) \\
            & + \frac{\tau(\chi_0)}{\phi(p)}\sum_{\substack{\psi \Mod{p} \\ \psi \not= \psi_0}} \tau\left(\overline{\psi}\right) \psi(a) R_{f\otimes\psi, q}(q^{-s}) \Lambda_{f\otimes \psi}(s)
        \end{split}
    \end{equation*}
    continues to a holomorphic function in $\{s \in \C : \Re(s) \leq 0\} \setminus \frac{1}{2}\Z$.
\end{lemma}

\begin{proof}
    By \eqref{additive_char_expansion_eq} we have
    \begin{equation*}
        \Delta_{f, a, p}\left(s, \frac{b}{q}\right) = \frac{q-1}{\phi(q)}\Delta_{f, a, p}(s) + \frac{q}{\phi(q)}\tau(\chi_0) \Delta_{f, a, p}(s, \chi_0) + \frac{1}{\phi(q)}\sum_{\substack{\chi \Mod{q} \\ \chi \not= \chi_0}} \tau\left(\overline{\chi}\right) \chi(b) \Delta_{f, a, p}(s, \chi).
    \end{equation*}
    For non-trivial $\chi \pmod{q}$, we can twist \eqref{Delta_additive_twist_decomposition_eq} by $\chi$ to get
    \begin{equation*}
        \begin{split}
            \Delta_{f, a, p}(s, \chi) &= \left(\frac{p-1}{\phi(p)} + \frac{p}{\phi(p)}\tau(\psi_0)P_{f, p}(p^{-s}\chi(p)) \right) \Delta_f(s, \chi) + \frac{1}{\phi(p)}\sum_{\substack{\psi \Mod{p} \\ \psi \not= \psi_0}} \tau\left(\overline{\psi}\right) \psi(a) \Delta_f(s, \psi \chi) \\
            &= \left(\frac{p-1}{\phi(p)} + \frac{p}{\phi(p)}\tau(\psi_0)P_{f\otimes\chi, p}(p^{-s}) \right) \Delta_{f\otimes\chi}(s) + \frac{1}{\phi(p)}\sum_{\substack{\psi \Mod{p} \\ \psi \not= \psi_0}} \tau\left(\overline{\psi}\right) \psi(a) \Delta_{f\otimes\psi\chi}(s),
        \end{split}
    \end{equation*}
    as $\chi\pmod{q}$ and $\psi\chi \pmod{pq}$ are primitive characters. This shows that $\Delta_{f, a, p}(s, \chi)$ is holomorphic in $\{s \in \C : \Re(s) \leq 0\} \setminus \frac{1}{2}\Z$, since this is the case for each of $\Delta_{f\otimes\chi}(s)$ and $\Delta_{f\otimes \psi\chi}(s)$ due to \autoref{Delta_additive_twist_pole_location_lemma}. The same property also holds for $\Delta_{f, a, p}(s)$ by the same lemma, so we are left with analyzing $\Delta_{f, a, p}(s, \chi_0)$. Since $\chi_0(p)=1$, once again by \eqref{Delta_additive_twist_decomposition_eq} we have
    \begin{equation*}
        \Delta_{f, a, p}(s, \chi_0) = \left(\frac{p-1}{\phi(p)} + \frac{p}{\phi(p)}\tau(\psi_0)P_{f, p}(p^{-s}) \right) \Delta_f(s, \chi_0) + \frac{1}{\phi(p)}\sum_{\substack{\psi \Mod{p} \\ \psi \not= \psi_0}} \tau\left(\overline{\psi}\right) \psi(a) \Delta_{f\otimes\psi}(s, \chi_0).
    \end{equation*}
    Now, \eqref{D_trivial_char_eq} gives
    \begin{equation*}
        \Delta_f(s, \chi_0) = P_{f, q}(q^{-s}) \Delta_f(s) - \frac{\phi(q)}{q} R_{f, q}(q^{-s}) \Lambda_f(s),
    \end{equation*}
    and analogously, since for non-trivial $\psi \pmod{p}$ the primitive form $f\otimes \psi$ has level $Np^2$ and $q \in Q(Np^2)$, 
    \begin{equation*}
        \Delta_{f\otimes \psi}(s, \chi_0) = P_{f\otimes \psi, q}(q^{-s}) \Delta_{f\otimes \psi}(s) - \frac{\phi(q)}{q} R_{f\otimes \psi, q}(q^{-s}) \Lambda_{f\otimes \psi}(s).
    \end{equation*}
    However, $\Delta_f(s)$ and $\Delta_{f\otimes\psi}(s)$ are both holomorphic in $\{s \in \C : \Re(s) \leq 0\} \setminus \frac{1}{2}\Z$, so the only remaining terms are the ones with the factors $R_{f, q}(q^{-s})$ and $R_{f\otimes \psi, q}(q^{-s})$. Plugging those back along our sequence of equations, we obtain the desired result.
    
\end{proof}

\subsection{Producing poles}

We are ready for the proof of the main result in this section. The first step is to show that the lack of poles for $G_{f, p}$ leads to a paradoxical analytic continuation for some of its additive twists. We compartmentalize this claim in the next lemma, which very closely follows the argument of \cite{Bo16}.

\begin{lemma}[No poles for $G_{f, p}$ implies continuation of additive twists]\label{continuation_of_additive_twists_lemma}
    Let $p \equiv 1 \Mod{N}$ be prime, and assume that $G_{f, p}(s)$ has no poles in $\Re(s) \in (0, 1)$. Then for any prime $q \nmid Np$,
    \begin{equation}\label{final_key_function_eq}
        \Delta_{f, 1, p}\left(s, \frac{1}{q}\right) -  \Delta_{f, -\overline{N}, p}\left(s, \frac{1}{q}\right)
    \end{equation}
    continues to a holomorphic function in $\C \setminus \frac{1}{2}\Z$.
\end{lemma}

\begin{proof}
    If $G_{f, p}(s)$ has no poles in $\Re(s) \in (0, 1)$, then by \autoref{Delta_additive_twist_pole_location_lemma} the only possible pole of $G_{f, p}(s)$ with $\Re(s) > -1/2$ is $s=0$, which can only occur if $k=1$.

    Let $0 < \eta < 1/2$. For $z \in \H$, define 
    \begin{equation*}
        \mathcal{I}^R(z) := \frac{1}{2\pi i } \int\limits_{\Re(s) = 1 + \eta} G_{f, p}(s)(-iz)^{-s - \frac{k-1}{2}} \dd s
    \end{equation*}
    and
    \begin{equation*}
        \mathcal{I}^L(z) := \frac{1}{2\pi i } \int\limits_{\Re(s) = -\eta} G_{f, p}(s)(-iz)^{-s - \frac{k-1}{2}} \dd s.
    \end{equation*}
    By Stirling's formula, the decomposition \eqref{Delta_additive_twist_decomposition_eq}, and the Phragm{\'e}n–Lindel{\"o}f principle, we see that $G_{f, p}(s)$ is rapidly decaying in vertical strips, so we can shift contours. Since we are assuming that $G_{f, p}(s)$ has no poles in $\Re(s) \in (0, 1)$, and it has a pole at $s=0$ only if $k=1$, we get
    \begin{equation}\label{G_contour_shift_eq}
        \mathcal{I}^L(z) + \Res_{s=0} G_{f, p}(s) = \mathcal{I}^R(z).
    \end{equation}
    
    Observe that
    \begin{equation*}
        \begin{split}
            \mathcal{I}^R(z) &= \frac{1}{2\pi i } \int\limits_{\Re(s) = 1 + \eta} \left( \Delta_{f, 1, p}(s) - \Delta_{f, -\overline{N}, p}(s) \right) (-iz)^{-s - \frac{k-1}{2}} \dd s \\
            &= I^R_{f, 1, p}(z) - I^R_{f, -\overline{N}, p}(z)
        \end{split}
    \end{equation*}
    in the notation of \autoref{mellin_computation_lemma}. Similarly, we have
    \begin{equation*}
        \begin{split}
            \mathcal{I}^L(z) + \Res_{s=0}G_{f, p}(s) =\ & \frac{1}{2\pi i } \int\limits_{\Re(s) = - \eta} \left( \Delta_{f, 1, p}(s) - \Delta_{f, -\overline{N}, p}(s) \right) (-iz)^{-s - \frac{k-1}{2}} \dd s + \Res_{s=0}G_{f, p}(s) \\
            =\ & \left(I^L_{f, 1, p}(z) + \Res_{s=0} \Delta_{f, 1, p}(s)\right) - \left(I^L_{f, -\overline{N}, p}(z) + \Res_{s=0} \Delta_{f, -\overline{N}, p}(s)\right).
        \end{split}
    \end{equation*}
    Therefore, \eqref{G_contour_shift_eq} becomes
    \begin{equation}\label{pre_mellin_funct_eq}
         \left(I^L_{f, 1, p}(z) + \Res_{s=0} \Delta_{f, 1, p}(s)\right) - \left(I^L_{f, -\overline{N}, p}(z) + \Res_{s=0} \Delta_{f, -\overline{N}, p}(s)\right) = I^R_{f, 1, p}(z) - I^R_{f, -\overline{N}, p}(z).
    \end{equation}
    
    We now set $z = \alpha + iy$, with $\alpha \in \Q^\times$ and $y > 0$, and perform a truncated Mellin transform along $y$. More precisely, consider
    \begin{equation*}
        \mathcal{R}(s):= \int_0^\frac{|\alpha|}{4} \left(I^R_{f, 1, p}(\alpha+iy) - I^R_{f, -\overline{N}, p}(\alpha+iy)\right) y^{s+\frac{k-1}{2}}\frac{dy}{y}.
    \end{equation*}
    Applying \autoref{analytic_continuation_lemma} we conclude that
    \begin{equation}\label{final_trunc_mellin_LHS_eq}
        \mathcal{R}(s) - \left(\Delta_{f, 1, p}(s, \alpha) - \Delta_{f, -\overline{N}, p}\left(s, \alpha\right)\right)
    \end{equation}
    continues to an entire function of $s$. Similarly, let
    \begin{equation*}
        \mathcal{L}(s) :=\int_0^\frac{|\alpha|}{4} \left(\left(I^L_{f, 1, p}(\alpha+iy) + \Res_{s=0} \Delta_{f, 1, p}(s)\right) - \left(I^L_{f, -\overline{N}, p}(\alpha+iy) + \Res_{s=0} \Delta_{f, -\overline{N}, p}(s)\right)\right) y^{s+\frac{k-1}{2}}\frac{dy}{y}.
    \end{equation*}
    By \autoref{analytic_continuation_lemma} we conclude that, for any $M \in \Z_{\geq 0}$,
    \begin{equation}\label{final_trunc_mellin_RHS_eq}
        \begin{split}
            \mathcal{L}(s) - \rn_f (i \sgn(\alpha))^k (Np^2\alpha^2)^{s-\frac{1}{2}}  \sum_{m=0}^{M-1}  (iNp^2\alpha)^{m} \binom{s+m-\frac{k+1}{2}}{m} \\
            \times \left(\Delta_{\overline{f}, -\overline{N}, p}\left(s+m, -\frac{1}{Np^2\alpha}\right) - \Delta_{\overline{f}, 1, p}\left(s+m, -\frac{1}{Np^2\alpha}\right) \right)
        \end{split}
    \end{equation}
    continues to a meromorphic function in $\{s \in \C : \Re(s) > 1-M\}$ whose poles in that region can only be at $s = \frac{1-k}{2} - n$ for $n \in \Z_{\geq 0}$. Here we used the fact that $\xi(p) = 1$, as $p \equiv 1 \pmod{N}$.
    
    Since $\mathcal{L}(s) = \mathcal{R}(s)$ due to \eqref{pre_mellin_funct_eq}, we conclude from \eqref{final_trunc_mellin_LHS_eq} and \eqref{final_trunc_mellin_RHS_eq} that
    \begin{equation}\label{final_general_delta_relation_eq}
        \begin{split}
            &\Delta_{f, 1, p}(s, \alpha) - \Delta_{f, -\overline{N}, p}(s, \alpha) - \rn_f (i \sgn(\alpha))^k (Np^2\alpha^2)^{s-\frac{1}{2}} \sum_{m=0}^{M-1} (iNp^2\alpha)^{m} \\
            \times \ & \binom{s+m-\frac{k+1}{2}}{m} \left(\Delta_{\overline{f}, -\overline{N}, p}\left(s+m, -\frac{1}{Np^2\alpha}\right) - \Delta_{\overline{f}, 1, p}\left(s+m, -\frac{1}{Np^2\alpha}\right) \right)
        \end{split}
    \end{equation}
    continues to a holomorphic function in $\{s \in \C : \Re(s) > 1-M\} \setminus \frac{1}{2}\Z$.
    
    Fix $b \in (\Z \slash Np^2\Z)^\times$. Let $q_1, q_2, \dots, q_M$ be distinct primes satisfying $q_j \equiv b \pmod{Np^2}$ for all $1\leq j\leq M$, and let $m_0$ be an integer satisfying $0 \leq m_0 \leq M-1$. Setting $\alpha = \frac{1}{q_j}$, \eqref{final_general_delta_relation_eq} shows that
    \begin{equation}\label{intermediate_delta_relation_eq}
        \begin{split}
            &\left(\frac{Np^2}{q_j^2}\right)^{\frac{1}{2}-s} \left(\Delta_{f, 1, p}\left(s, \frac{1}{q_j}\right) -  \Delta_{f, -\overline{N}, p}\left(s, \frac{1}{q_j}\right)\right) - \rn_f i^k \sum_{m=0}^{M-1} \left(\frac{iNp^2}{q_j}\right)^{m} \\
            \times \ & \binom{s+m-\frac{k+1}{2}}{m} \left(\Delta_{\overline{f}, -\overline{N}, p}\left(s+m, -\frac{b}{Np^2}\right) - \Delta_{\overline{f}, 1, p}\left(s+m, -\frac{b}{Np^2}\right) \right)
        \end{split}
    \end{equation}
    continues to a holomorphic function in $\{s \in \C : \Re(s) > 1-M\} \setminus \frac{1}{2}\Z$. By the Vandermonde determinant, we can find $c_1, c_2, \dots, c_M \in \Q$ such that for every $m\in \Z$ with $0\leq m\leq M-1$,
    \begin{equation*}
        \sum_{j=1}^M c_j q_j^{-m} = 
        \begin{cases}
            1 & \text{if } m = m_0, \\
            0 & \text{if } m \not = m_0.
        \end{cases}
    \end{equation*}
    Summing \eqref{intermediate_delta_relation_eq} for each $q_j$ with weight $c_j$, for $1\leq j\leq M$, it follows that
    \begin{equation}\label{final_specific_delta_relation_eq}
        \begin{split}
            &\rn_f i^k (iNp^2)^{m_0} \binom{s+m_0 -\frac{k+1}{2}}{m_0} \left(\Delta_{\overline{f}, -\overline{N}, p}\left(s+m_0, -\frac{b}{Np^2}\right) - \Delta_{\overline{f}, 1, p}\left(s+m_0, -\frac{b}{Np^2}\right) \right) \\
            - & \sum_{j=1}^M c_j \left(\frac{Np^2}{q_j^2}\right)^{\frac{1}{2}-s} \left(\Delta_{f, 1, p}\left(s, \frac{1}{q_j}\right) -  \Delta_{f, -\overline{N}, p}\left(s, \frac{1}{q_j}\right)\right)
        \end{split}
    \end{equation}
    continues to a holomorphic function in $\{s \in \C : \Re(s) > 1-M\} \setminus \frac{1}{2}\Z$.

    Now, observe that both $\Delta_{f, 1, p}\left(s, \frac{1}{q_j}\right)$ and  $\Delta_{f, -\overline{N}, p}\left(s, \frac{1}{q_j}\right)$ are holomorphic in $\{s \in \C : \Re(s) < 0\} \setminus \frac{1}{2}\Z$. Indeed, this follows from \autoref{Delta_hardcore_pole_location_lemma} and the fact that for a non-trivial character $\psi \pmod{p}$, the poles of $R_{f, q_j}(q_j^{-s})$ and $R_{f\otimes\psi, q_j}(q_j^{-s})$ satisfy $\Re(s) = 0$, since $\lambda_{f\otimes \psi}(q_j) = \lambda_f(q_j) \psi(q_j)$ and $|\lambda_f(q_j)| \leq 2$ by Deligne's bound. Therefore, \eqref{final_specific_delta_relation_eq} implies that
    \begin{equation*}
        \Delta_{\overline{f}, -\overline{N}, p}\left(s, -\frac{b}{Np^2}\right) - \Delta_{\overline{f}, 1, p}\left(s, -\frac{b}{Np^2}\right)
    \end{equation*}
    continues to a holomorphic function in $\{s \in \C : 1-M + m_0 < \Re(s) < m_0 \}\setminus \frac{1}{2}\Z$. Since $M \in \Z_{\geq 0}$ and $0 \leq m_0 \leq M-1$ are arbitrary, we conclude that it indeed continues to a holomorphic function in $\C \setminus \frac{1}{2}\Z$. Finally, this in conjunction with \eqref{intermediate_delta_relation_eq} shows that
    the desired function \eqref{final_key_function_eq} continues to a holomorphic function in $\C \setminus \frac{1}{2}\Z$, for any prime $q \equiv b \pmod{Np^2}$. Since we can choose the congruence class $b \in (\Z\slash Np^2 \Z)^\times$ arbitrarily, the result holds for any prime $q \nmid Np$.
    
\end{proof}

To finish, we produce poles for some choice of twists from the previous lemma. Some extra care must be taken compared with \cite{Bo16}, since instead of $\Delta_f$ we have a linear combination $G_{f, p}$ of such terms. Fortunately at this point we can rule out complete cancellation of the poles due to their explicit nature, arising from simple zeros of local factors.

\begin{proof}[Proof of \autoref{pole_existence_prop}]
    Assume by contradiction that $G_{f, p}(s)$ has no poles in $\Re(s) \in (0, 1)$. Then \autoref{continuation_of_additive_twists_lemma} shows that for each prime $q\nmid Np$,  \eqref{final_key_function_eq} continues to a holomorphic function in $\C \setminus \frac{1}{2}\Z$.
    
    Let $\chi_0 \pmod{q}$ and $\psi_0 \pmod{p}$ denote the trivial characters, and observe that $\tau(\chi_0) = \tau(\psi_0) = -1$. Applying \autoref{Delta_hardcore_pole_location_lemma} to each term of \eqref{final_key_function_eq} we verify that
    \begin{equation}\label{final_key_function_with_poles_eq}
            \frac{1}{p-1} \sum_{\substack{\psi \Mod{p} \\ \psi \not= \psi_0}} \tau\left(\overline{\psi}\right) \left(\psi(1) - \psi\left( -\overline{N}\right)\right) R_{f\otimes\psi, q}(q^{-s}) \Lambda_{f\otimes\psi}(s)
    \end{equation}
    continues to a holomorphic function in $\{s\in\C : \Re(s) \leq 0\} \setminus \frac{1}{2}\Z$, so in particular it has no poles $s \not= 0$ with $\Re(s) = 0$.
    
    Observe that for any $c \in (\Z \slash p \Z)^\times$ we have
    \begin{equation}\label{PNT_in_AP_eq}
        \sum_{\substack{r\leq x \text{ prime} \\ r \equiv c \Mod{p}}} |\lambda_f(r)|^2 \sim \frac{1}{\phi(p)} \frac{x}{\log{x}} \quad \quad \text{as } x \to \infty
    \end{equation}
    by Rankin-Selberg (see for instance \cite[Lemma 1]{LM04} for details when $f$ has trivial nebentypus), as $f \otimes \psi$ is orthogonal to $f$ for each non-trivial $\psi \pmod{p}$, since it is a primitive form of level $Np^2$. From now on we assume that $q \in \mathcal{Q}_{f, p} := \left\{ r \text{ prime} : r \equiv 1 \pmod{p}, r \nmid N, \text{ and }|\lambda_f(r)| < 2\right\}$. Observe that $\mathcal{Q}_{f, p}$ is an infinite set, by \eqref{PNT_in_AP_eq}. 

    Since $q \equiv 1 \pmod{p}$, for any non-trivial $\psi\pmod{p}$ we have $R_{f\otimes\psi, q}(q^{-s}) = R_{f, q}(q^{-s}\psi(q)) = R_{f, q}(q^{-s})$. But
    \begin{equation}\label{R_and_P_relation_eq}
        R_{f, q}(q^{-s}) = -\frac{q}{\phi(q)} P_{f, q}(q^{-s}) \left(\frac{\left(P_{f, q}(q^{-s})\right)'}{P_{f, q}(q^{-s})}\right)',
    \end{equation}
    where the derivatives are with respect to $s$, so the poles of $R_{f, q}(q^{-s})$ are precisely at the simple zeros of $P_{f, q}(q^{-s}) = 1 - \lambda_f(q)q^{-s} +\xi(q) q^{-2s} =: (1-\alpha_{f}(q)q^{-s})(1-\beta_f(q) q^{-s})$. We chose $q$ with $|\lambda_f(q)| < 2$, so $|\alpha_f(q)| = |\beta_f(q)| = 1$ and $\alpha_f(q) \not= \beta_f(q)$. Therefore, all the zeros of $P_{f, q}(q^{-s})$ are simple and satisfy $\Re(s) = 0$.
    
    Choose $t \in \R^\times$ such that $q^{it} = \alpha_f(q)$, so $P_{f, q}(q^{-it}) = 0$ and \eqref{R_and_P_relation_eq} gives 
    \begin{equation*}
        \begin{split}
            \Res_{s=it} R_{f, q}(q^{-s}) &= \left.\frac{q}{\phi(q)} (P_{f, q}(q^{-s}))' \right\vert_{s=it} = \frac{q^{1-it}  \log{q}}{q-1} \left(\lambda_f(q) - 2 \xi(q) q^{-it}\right) \\
            &= \frac{q  \log{q}}{q-1} \overline{\alpha_f(q)} \left(\alpha_f(q) - \beta_f(q)\right) \not= 0, 
        \end{split}
    \end{equation*}
    as $\alpha_f(q) \beta_f(q) = \xi(q)$. We now take residues of \eqref{final_key_function_with_poles_eq} at $s=it \not=0$ to obtain 
    \begin{equation*}
        \frac{1}{p-1} \sum_{\substack{\psi \Mod{p} \\ \psi \not= \psi_0}} \tau\left(\overline{\psi}\right) \left(\psi(1) - \psi\left( -\overline{N}\right)\right) \Lambda_{f\otimes\psi}(it) \cdot \Res_{s=it} R_{f, q}(q^{-s}) = 0.
    \end{equation*}
    But $\Res_{s=it} R_{f, q}(q^{-s}) \not= 0$ as we saw above, so in fact using \eqref{lambda_additive_twist_char_expansion_eq} we get
    \begin{equation}
        0 = \frac{1}{p-1} \sum_{\substack{\psi \Mod{p} \\ \psi \not= \psi_0}} \tau\left(\overline{\psi}\right) \left(\psi(1) - \psi\left( -\overline{N}\right)\right) \Lambda_{f\otimes\psi}(it) = \Lambda_f\left(it, \frac{1}{p}\right) - \Lambda_f\left(it, -\frac{\overline{N}}{p}\right)
    \end{equation}
    for any $t \in \mathcal{T}_{f, q} := \left\{\frac{\theta_f(q) + 2\pi n}{\log{q}} : n\in \Z\right\} \setminus \{0\}$, where $\theta_f(q) \in [0, 2\pi)$ is defined by $\alpha_f(q) = e^{i\theta_f(q)}$. Since this holds for any $q \in \mathcal{Q}_{f, p}$ and $\bigcup_{q \in \mathcal{Q}_{f, p}} \mathcal{T}_{f, q}$ is dense in $\R$ (as $\mathcal{Q}_{f, p}$ is infinite), we conclude by analytic continuation that
    \begin{equation*}
        \Lambda_f\left(s, \frac{1}{p}\right) = \Lambda_f\left(s, -\frac{\overline{N}}{p}\right)
    \end{equation*}
    for every $s \in \C$. This is a contradiction, as we can compare the coefficients of the respective Dirichlet series expansions in $\Re(s) > 1$ and they do not match. For instance, \eqref{PNT_in_AP_eq} shows that there is a prime $r \equiv 1 \pmod{p}$ such that $\lambda_f(r) \not=0$, hence the $r$-th coefficients in the corresponding Dirichlet series expansions are $\lambda_f(r) e\left(1/p\right)$ and $\lambda_f(r) e\left(-\overline{N}/p\right)$, which are distinct since $-\overline{N} \not\equiv 1\pmod{p}$, as $p > N+1$  by hypothesis. A standard argument using Perron's formula then gives the desired contradiction, so we conclude that $G_{f, p}(s)$ has at least one pole in $\Re(s) \in (0, 1)$, as desired.
    
\end{proof}


\section{Location of poles of \texorpdfstring{$H_{f, \alpha}$}{H {f, alpha}}}

In this section we will show that if $H_{f, \alpha}(s)$ has a pole in $\Re(s) \in \left[\frac{1}{2}, 1\right)$ for some $\alpha \in \Q^\times$, then $\Lambda_f$ must have many simple zeros. This will be enough to prove our main results, since \autoref{H_large_pole_existence_prop} guarantees the existence of such a pole for some $\alpha$ in the case of either $f$ or $\overline{f}$, but $\Lambda_f$ and $\Lambda_{\overline{f}}$ have the same number of simple zeros, by the functional equation.

\subsection{From poles of \texorpdfstring{$H_{f, \alpha}$}{H {f, alpha}} to simple zeros of \texorpdfstring{$\Lambda_f$}{Lambda f}}

Denote $S_f(z) := S_{f, 1, 1}(z)$ for $z\in\H$, as in the introduction. As we have described before, the basic mechanism uses \eqref{S_abs_val_condition_eq} to show that $S_f$ cannot be always small if $H_{f, \alpha}$ has a pole of large real part. The next lemma provides a more direct link between the quantity in \eqref{S_abs_val_condition_eq} and simple zeros of $\Lambda_f$ (i.e.\ poles of $\Delta_f$ in the critical strip). It is essentially contained in \cite[Lemma 3.2]{BMN19}, but we provide a proof for completeness.

\begin{lemma}[Bounding the truncated Mellin transform of $S_f$]\label{S_bound_lemma}
    Let $\eta > 0$ be fixed. For any $\sigma \in [\eta, 2]$ and $\alpha \in \Q^\times$,
    \begin{equation*}
        \int_0^\frac{|\alpha|}{4} |S_f(\alpha + iy)| y^{\sigma + \frac{k-1}{2}}\frac{dy}{y} \ll_{f, \alpha, \eta} \sum_{\substack{\rho = \beta + i\gamma \\ \text{ a pole of } \Delta_f \\ \text{with }\beta >0}} \left|\Lambda_f'(\rho)\right| e^\frac{\pi |\gamma|}{2}(1+|\gamma|)^{-\sigma-\frac{k-1}{2}}.
    \end{equation*}
\end{lemma}

\begin{proof}
    We have
    \begin{equation}\label{integral_abs_val_bound_eq}
        \int_0^\frac{|\alpha|}{4} |S_f(\alpha + iy)| y^{\sigma + \frac{k-1}{2}}\frac{dy}{y} \leq \sum_{\substack{\Re(\rho) \in (0, 1)}} \int_0^\frac{|\alpha|}{4} \left|\Res_{s=\rho}\Delta_f(s)\right|\cdot\left|\left(y-i\alpha\right)^{-\rho-\frac{k-1}{2}}\right| y^{\sigma + \frac{k-1}{2}}\frac{dy}{y},
    \end{equation}
    where we can exchange the order of summation and integration by Tonelli's theorem. Let $\rho = \beta + i\gamma$ be a pole of $\Delta_f$ with $\beta \in (0, 1)$, and denote $\tau := 1+|\gamma|$. Then since $\Lambda_f(\rho) = 0$, observe that
    \begin{equation*}
        \Res_{s=\rho}\Delta_f(s) = -\Gamma_\C\left(\rho + \frac{k-1}{2}\right) L_f'(\rho) = -\Lambda_f'(\rho),
    \end{equation*}
    and for $0 \leq y \leq \frac{|\alpha|}{4}$,
    \begin{equation*}
        \left|\left(y-i\alpha\right)^{-\rho - \frac{k-1}{2}} \right| = |y-i\alpha|^{-\beta-\frac{k-1}{2}} e^{\gamma \arctan\left(-\frac{\alpha}{y}\right)} \ll_{f, \alpha} e^{-\gamma \arctan\left(\frac{\alpha}{y}\right)} = e^{\gamma \sgn(\alpha) \left(\arctan\left(\frac{y}{|\alpha|}\right) - \frac{\pi}{2} \right)}.
    \end{equation*}
    Therefore, we conclude that the RHS of \eqref{integral_abs_val_bound_eq} is
    \begin{equation*}
        \ll_{f, \alpha} \sum_{\substack{\rho = \beta + i\gamma \\ \text{ a pole of } \Delta_f \\ \text{with }\beta \in (0, 1)}} \left|\Lambda_f'(\rho)\right| \cdot \int_0^\frac{|\alpha|}{4}  e^{\gamma \sgn(\alpha) \left(\arctan\left(\frac{y}{|\alpha|}\right)-\frac{\pi}{2}\right)} y^{\sigma + \frac{k-1}{2}}\frac{dy}{y}.
    \end{equation*}
    Using $\gamma \sgn(\alpha) \left(\arctan\left(\frac{y}{|\alpha|}\right)-\frac{\pi}{2}\right) \leq -|\gamma|\left(\arctan\left(\frac{y}{|\alpha|}\right) - \frac{\pi}{2}\right) \leq -\frac{|\gamma|y}{2|\alpha|} + \frac{\pi|\gamma|}{2}$, since $\arctan(x) \geq \frac{x}{2}$ for $0 \leq x\leq \frac{1}{4}$, we have
    \begin{equation*}
        \begin{split}
            \int_0^\frac{|\alpha|}{4}  e^{\gamma \sgn(\alpha) \left(\arctan\left(\frac{y}{|\alpha|}\right)-\frac{\pi}{2}\right)} y^{\sigma + \frac{k-1}{2}}\frac{dy}{y} \leq e^{\frac{\pi|\gamma|}{2}} \int_0^\frac{|\alpha|}{4}  e^{-\frac{|\gamma|y}{2|\alpha|}} y^{\sigma + \frac{k-1}{2}}\frac{dy}{y} \ll e^{\frac{\pi|\gamma|}{2}} \int_0^\frac{|\alpha|}{4}  e^{-\frac{\tau y}{2|\alpha|}} y^{\sigma + \frac{k-1}{2}}\frac{dy}{y} \\
            \leq e^{\frac{\pi|\gamma|}{2}} \int_0^\infty  e^{-\frac{\tau y}{2|\alpha|}} y^{\sigma + \frac{k-1}{2}}\frac{dy}{y} = e^{\frac{\pi|\gamma|}{2}} \left(\frac{2|\alpha|}{\tau}\right)^{\sigma + \frac{k-1}{2}}\Gamma\left(\sigma + \frac{k-1}{2}\right) \ll_{f, \alpha, \eta} e^{\frac{\pi|\gamma|}{2}} \tau^{-\sigma-\frac{k-1}{2}},
        \end{split}
    \end{equation*}
    so the desired result follows.
    
\end{proof}

For the case of general level, we will apply \autoref{S_bound_lemma} in conjunction with a pointwise bound coming from subconvexity (we will see how to improve this for $N=1$ in the next section).

\begin{lemma}[Weyl subconvexity for $L'_f$ {\cite[Lemma 3.1]{BMN19}}]\label{weyl_subconvexity_lemma}
    If $\rho = \beta+i\gamma$ is a zero of $\Lambda_f$, then
    \begin{equation*}
        \Lambda'_f(\rho) \ll_{f, \eps} (1+|\gamma|)^{\frac{k}{2} + \frac{1}{3}\left|\beta - \frac{1}{2} \right| - \frac{1}{6} + \eps}e^{-\frac{\pi|\gamma|}{2}}
    \end{equation*}
    for any $\eps > 0$.
\end{lemma}

\begin{proof}[Proof (sketch)]
    Follows from the Weyl subconvex bound $L_f\left(\frac{1}{2}+it\right) \ll_{f, \eps} (1+|t|)^{\frac{1}{3}+\eps}$ of \cite{Weyl19}, and a standard argument using Cauchy's formula combined with the Phragm{\'e}n–Lindel{\"o}f principle, the functional equation, and Stirling's formula. See the reference for details.
    
\end{proof}

\begin{remark}\label{subconvexity_exponent_rmk}
    If $\mu \in \left[0, \frac{1}{2}\right]$ and we had a subconvexity bound of the form $L_f\left(\frac{1}{2}+it\right) \ll_{f, \eps} (1+|t|)^{\mu + \eps}$ for all $\eps>0$, then \autoref{weyl_subconvexity_lemma} would become $\Lambda_f'(\rho) \ll_{f, \eps} (1+|\gamma|)^{\frac{k}{2} + (1-2\mu)\left(\left| \beta - \frac{1}{2}\right| - \frac{1}{2}\right) + \eps} e^{-\frac{\pi|\gamma|}{2}}$ for any $\eps > 0$. The given result corresponds to $\mu = \frac{1}{3}$.
\end{remark}

For a meromorphic function $h$ on $\{ s\in\C : \Re(s) > 1 \}$, let
\begin{equation*}
    \Theta(h) := \inf\left\{\theta \geq 0 : h \text{ continues analytically to }\{s\in\C : \Re(s) > \theta\}\right\}.
\end{equation*}
Furthermore, let
\begin{equation*}
    \begin{split}
        \theta_f :=\ & \sup\left( \{0\} \cup \left\{ \Re(\rho), 1-\Re(\rho) : \rho \text{ is a pole of } \Delta_f  \right\} \right)\\
        =\ & \sup\left( \{0\} \cup \left\{ \Re(\rho) : \rho \text{ is a simple zero of } \Lambda_f \text{ or } \Lambda_{\overline{f}} \right\} \right).
    \end{split}
\end{equation*}
Then \autoref{S_bound_lemma} and \autoref{weyl_subconvexity_lemma} can be combined into the following result, which is a particular case of \cite[Proposition 3.3]{BMN19}. We again provide the proof for completeness.

\begin{proposition}[General bound for $N_f^s$]\label{initial_zero_bound_prop}
    Let $\alpha \in \Q^\times$. If $\Theta(H_{f, \alpha}) > 0$, then $\theta_f \geq \frac{1}{2}$ and
    \begin{equation*}
        N_f^s(T) = \Omega\left(T^{\frac{1}{3}(1-\theta_f) + \Theta(H_{f, \alpha}) - \frac{1}{2} - \eps}\right)
    \end{equation*}
    for any $\eps > 0$.
\end{proposition}

\begin{proof}
    Let $\beta_n + i \gamma_n$ run through the poles of $\Delta_f$ with $\beta_n >0$, in increasing order of $|\gamma_n|$. For $\sigma \in (0, 1]$, \autoref{S_bound_lemma} and \autoref{weyl_subconvexity_lemma} give
    \begin{equation}\label{S_integral_lemma_eq}
        \int_0^\frac{|\alpha|}{4} |S_f(\alpha+iy)| y^{\sigma + \frac{k-1}{2}}\frac{dy}{y} \ll_{f, \alpha, \sigma, \eps} \sum_{n=1}^\infty (1+|\gamma_n|)^{\frac{1}{3}\left|\beta_n-\frac{1}{2}\right| + \frac{1}{3} - \sigma + \eps}
    \end{equation}
    for any $\eps > 0$. If $\Theta(H_{f, \alpha}) > 0$, set $\sigma = \Theta(H_{f, \alpha}) -\eps$, where $0 < \eps < \Theta(H_{f, \alpha})$ is arbitrary. Then \autoref{pole_detection_prop} implies that \eqref{S_integral_lemma_eq} diverges, so in particular $\Delta_f$ has infinitely many poles $\beta_n+i\gamma_n$ with $\beta_n >0$, and therefore $\theta_f \geq \frac{1}{2}$.
    
    Now assume by contradiction that $N_f^s(T) = o\left(T^{\frac{1}{3}(1-\theta_f)+\Theta(H_{f, \alpha}) - \frac{1}{2}-3\eps}\right)$ for some $0 < \eps < \Theta(H_{f, \alpha})$. Then by \eqref{S_integral_lemma_eq}, since $\left|\beta_n - \frac{1}{2} \right|\leq \theta_f - \frac{1}{2}$, we get
    \begin{equation*}
        \begin{split}
            \infty = \int_0^\frac{|\alpha|}{4} |S_f(\alpha+iy)| y^{\sigma + \frac{k-1}{2}}\frac{dy}{y} &\ll_{f, \alpha, \eps} \sum_{n=1}^\infty (1+|\gamma_n|)^{\frac{1}{3}\theta_f+\frac{1}{6}- \Theta(H_{f, \alpha}) + 2\eps} \\
            &\ll_{f} 1 + \int_1^\infty t^{\frac{1}{3}\theta_f+\frac{1}{6}- \Theta(H_{f, \alpha}) + 2\eps} \dd N_f^s(t) \\
            &\ll_f 1 + \int_1^\infty t^{\frac{1}{3}\theta_f+\frac{1}{6} - \Theta(H_{f, \alpha})+2\eps - 1} N^s_f(t) \dd t \\
            &= 1 + \int_1^\infty o\left(t^{-1-\eps}\right) \dd t < \infty,
        \end{split}
    \end{equation*}
    which is a contradiction.
    
\end{proof}

\begin{remark}\label{subconvexity_rmk}
    Assuming a subconvexity exponent $\mu \in \left[0, \frac{1}{2}\right]$ as in \autoref{subconvexity_exponent_rmk}, the result of \autoref{initial_zero_bound_prop} becomes $N_f^s(T) = \Omega\left(T^{(1-2\mu)(1-\theta_f) + \Theta(H_{f, \alpha}) - \frac{1}{2} - \eps}\right)$ for any $\eps>0$.
\end{remark}

Observe that \autoref{initial_zero_bound_prop} fails to give a power of $T$ (even if the subconvexity exponent were to be improved) if $\theta_f = 1$ and $\Theta(H_{f, \alpha}) = \frac{1}{2}$, which cannot be ruled out with what we have done so far. However, we will use this proposition for the case of $\theta_f$ sufficiently far from $1$, where it gives a good bound.

\begin{corollary}[Main result for $\theta_f$ away from $1$]\label{small_theta_cor}
    We have $\theta_f \geq \frac{1}{2}$ and 
    \begin{equation*}
        N_f^s(T) = \Omega\left(T^{\frac{1}{3}(1-\theta_f) - \eps}\right)
    \end{equation*}
    for any $\eps > 0$.
\end{corollary}

\begin{proof}
    By the functional equation $\Lambda_f(s) = \rn_f N^{\frac{1}{2}-s} \Lambda_{\overline{f}}(1-s)$, we have $N_f^s(T) = N_{\overline{f}}^s(T)$ and $\theta_f = \theta_{\overline{f}}$. By \autoref{H_large_pole_existence_prop}, there is $\alpha_f \in \Q^\times$ such that $\max\left\{\Theta(H_{f, \alpha_f}), \Theta(H_{\overline{f}, \alpha_f})\right\} \geq \frac{1}{2}$. Then applying \autoref{initial_zero_bound_prop} to either $f$ or $\overline{f}$ gives the desired result.
    
\end{proof}

\subsection{Improvements for \texorpdfstring{$\theta_f$}{theta f} close to \texorpdfstring{$1$}{1}}

If $\theta_f$ is close to $1$, then either $\Delta_f$ or $\Delta_{\overline{f}}$ must have a pole $\rho$ with real part close to $1$. We will show that if for instance that is the case for $\Delta_f$, then there exists $\alpha \in \Q^\times$ such that $H_{f, \alpha}$ also has a pole at $\rho$, so \autoref{initial_zero_bound_prop} gives a much stronger result than before. The main tool for showing such a pole transference will be a certain zero density estimate, which we introduce now. 

For a primitive form $g \in S_k(\Gamma_1(M))$, $\beta \in \R$, and $T \geq 0$,  let
\begin{equation}\label{number_zero_def_eq}
    N_g(\beta, T) := \left| \{ s \in \C : \Re(s) \geq \beta, |\Im(s)| \leq T, \text{ and } L_g(s) = 0 \}\right|,
\end{equation}
where the zeros are counted with multiplicity.

\begin{lemma}[Zero density for twists close to the line $1$]\label{zero_density_lemma}
    Let $f \in S_k(\Gamma_1(N))$ be a primitive form. For each prime $p \equiv 1 \pmod{N}$, let $\psi_p \pmod{p}$ be an arbitrary non-trivial character modulo $p$. Then for any $T \geq 2$, $X \geq 2$, $\eps > 0$, and $\frac{3}{4} \leq \beta \leq 1$, we have
    \begin{equation*}
        \sum_{\substack{p \leq X \text{ prime} \\ p\equiv 1 \Mod{N}}} N_{f\otimes \psi_p}(\beta, T) \ll_{f, \eps, T} X^{4(1-\beta) + \eps} + X^{\frac{6(1-\beta)}{3\beta-1} + \eps}.
    \end{equation*}
\end{lemma}

\begin{proof}
    This follows directly from the more general result of \autoref{zero_density_prop} in \autoref{zero_density_appendix}. 
    
\end{proof}

Now, let $\kappa$ be such that $\frac{6(1-\kappa)}{3\kappa - 1} = 1$, i.e.\ $\kappa = \frac{7}{9}$. The important point is that $\kappa < 1$.

\begin{proposition}[Ruling out pole cancellation in $H_{f, \alpha}$ via zero density]\label{large_theta_pole_rule_out_prop}
    If $\Delta_f$ has a pole $\rho = \beta + i\gamma$ with $\beta > \kappa$, then there exists some $\alpha \in \Q^\times$ (depending on $f$ and $\rho$) such that $\rho$ is also a pole of $H_{f, \alpha}$. 
\end{proposition}

\begin{proof}
    We will show that there exists a prime $p$ satisfying $p \equiv 1\pmod{N}$ such that $\rho$ is a pole of either $H_{f, 1}$ or $H_{f, \frac{1}{p}}$, so we will be able to pick $\alpha = 1$ or $\alpha = \frac{1}{p}$. 
    
    Suppose by contradiction that $\rho$ is not a pole of either $H_{f, 1}$ or $H_{f, \frac{1}{p}}$ for any prime $p \equiv 1\pmod{N}$. By \eqref{first_H_diff_eq} we have
    \begin{equation*}
        p^{1-2s} H_{f, 1}(s) - H_{f, \frac{1}{p}}(s) = p^{1-2s}\Delta_f(s) - \Delta_{f, 1, p}(s) + R_{f, p}(p^{-s}) \Lambda_f(s),
    \end{equation*}
    and by assumption this meromorphic function does not have a pole at $s=\rho$. Observe that since $|\lambda_f(p)| \leq 2$ by Deligne's bound, the poles of $R_{f, p}(p^{-s}) \Lambda_f(s)$ all satisfy $\Re(s) = 0$, so $\rho$ is not a pole of $R_{f, p}(p^{-s}) \Lambda_f(s)$ (as $\kappa > 0$). Hence it also cannot be a pole of 
    \begin{equation*}
        \begin{split}
            p^{1-2s}\Delta_f(s) - \Delta_{f, 1, p}(s) = \left(p^{1-2s} -1 + \frac{p}{p-1}P_{f, p}(p^{-s}) \right) \Delta_f(s)  - \frac{1}{p-1}\sum_{\substack{\psi \Mod{p} \\ \psi \not= \psi_0}} \tau\left(\overline{\psi}\right) \Delta_{f\otimes\psi}(s),
        \end{split}
    \end{equation*}
    where $\psi_0 \pmod{p}$ denotes the trivial character, and we have used \eqref{Delta_additive_twist_decomposition_eq}. 
    
    Since $\xi(p) = 1$, a direct computation gives
    \begin{equation*}
        p^{1-2s} -1 + \frac{p}{p-1}P_{f, p}(p^{-s}) = \frac{p^{2-2s}-\lambda_f(p)p^{1-s}+1}{p-1} = \frac{1}{p-1}P_{f, p}(p^{1-s}).
    \end{equation*}
    Observe that $P_{f, p}(p^{1-\rho}) \not= 0$, as $\Re(1-\rho) = 1-\beta \not=0$, since $\beta <1$ by \cite{JS77}. Furthermore, 
    \begin{equation*}
        \Res_{s=\rho} \Delta_f(s) = -\Gamma_\C\left(\rho + \frac{k-1}{2}\right) L_f'(\rho)= -\Lambda_f'(\rho) \not= 0, 
    \end{equation*}
    as $\rho$ is a simple zero of $\Lambda_f$. We conclude that
    \begin{equation*}
        \sum_{\substack{\psi \Mod{p} \\ \psi \not= \psi_0}} \tau\left(\overline{\psi}\right) \cdot \Res_{s=\rho} \Delta_{f\otimes\psi}(s) = \Res_{s=\rho} P_{f, p}(p^{1-s}) \Delta_f(s) = -P_{f, p}(p^{1-\rho}) \Lambda_f'(\rho) \not= 0,
    \end{equation*}
    so there is at least one non-trivial character $\psi_p \pmod{p}$ such that $\Delta_{f\otimes\psi_p}$ has a pole at $\rho$, or in other words $\Lambda_{f\otimes\psi_p}$ has a simple zero at $\rho = \beta + i\gamma$. This holds for every prime $p \equiv 1 \pmod{N}$, so it follows that
    \begin{equation}\label{zeros_lower_bound_eq}
        \sum_{\substack{p \leq X \text{ prime} \\ p\equiv 1 \Mod{N}}} N_{f\otimes \psi_p}(\beta, 2+|\gamma|) \geq \pi(X; N, 1) \gg_f\frac{X}{\log{X}} 
    \end{equation}
    for every $X$ sufficiently large (in terms of $N$). However, applying \autoref{zero_density_lemma} we conclude that
    \begin{equation}\label{zeros_upper_bound_eq}
        \sum_{\substack{p \leq X \text{ prime} \\ p\equiv 1 \Mod{N}}} N_{f\otimes \psi_p}(\beta, 2+|\gamma|) \ll_{f, \eps, \rho} X^{4(1-\beta) + \eps} + X^{\frac{6(1-\beta)}{3\beta-1} + \eps}
    \end{equation}
    for every $X \geq 2$ and $\eps > 0$. Observe that both $\frac{6(1-x)}{3x-1}$ and $4(1-x)$ are strictly decreasing for $\frac{3}{4}\leq x \leq 1$, so since $\beta > \kappa$ and $\frac{6(1-\kappa)}{3\kappa-1} = 1$, while $4(1-\kappa) = \frac{8}{9} < 1$, we conclude that $4(1-\beta) + \eps < 1$ and $\frac{6(1-\beta)}{3\beta-1} + \eps < 1$ for $\eps > 0$ sufficiently small. But this is a contradiction, as \eqref{zeros_lower_bound_eq} and \eqref{zeros_upper_bound_eq} imply
    \begin{equation*}
        X^{4(1-\beta) + \eps} + X^{\frac{6(1-\beta)}{3\beta-1} + \eps} \gg_{f, \eps, \rho} \frac{X}{\log{X}}
    \end{equation*}
    for every $\eps>0$ and $X$ sufficiently large, which cannot hold for small $\eps > 0$ when $X \to \infty$. Therefore, the desired result follows by contradiction.
    
\end{proof}

\begin{corollary}[Main result for $\theta_f$ close to $1$]\label{large_theta_cor}
    If $\theta_f > \kappa$, then 
    \begin{equation*}
        N_f^s(T) = \Omega\left(T^{\frac{2}{3}\theta_f - \frac{1}{6} - \eps}\right)
    \end{equation*}
    for any $\eps > 0$.
\end{corollary}

\begin{proof}
    If $\theta_f > \kappa$, then for any given $0 < \eps < \theta_f - \kappa$, either $\Delta_f$ or $\Delta_{\overline{f}}$ must have a pole $\rho = \beta+i\gamma$ with $\beta > \theta_f-\eps$. Since $\theta_f-\eps > \kappa$, by \autoref{large_theta_pole_rule_out_prop} there exists some $\alpha \in \Q^\times$ (depending on $f$ and $\rho$) such that $\rho$ is also a pole of either $H_{f, \alpha}$ or $H_{\overline{f}, \alpha}$. Therefore,
    \begin{equation*}
        \max\left\{\Theta(H_{f, \alpha}), \Theta(H_{\overline{f}, \alpha}) \right\} \geq \beta > \theta_f-\eps.
    \end{equation*}
    
    Then we can use the relations $N_f^s(T) = N_{\overline{f}}^s(T)$ and $\theta_f = \theta_{\overline{f}}$ to get the desired result after applying \autoref{initial_zero_bound_prop} to either $f$ or $\overline{f}$, since $\eps > 0$ can be chosen arbitrarily small.
    
\end{proof}

\subsection{Obtaining a power bound}

The proof of our main theorem easily follows from what we have done so far.

\begin{proof}[Proof of \autoref{main_thm}]
    If $\theta_f > \kappa$, we apply \autoref{large_theta_cor} and observe that $\frac{2}{3}\theta_f-\frac{1}{6} > \frac{2}{3}\kappa-\frac{1}{6} = \frac{19}{54}$ to get
    \begin{equation*}
        N_f^s(T) = \Omega\left( T^{\frac{2}{3}\theta_f-\frac{1}{6} - \eps} \right) = \Omega\left( T^{\frac{19}{54} - \eps} \right) 
    \end{equation*}
    for any $\eps>0$, so in this case we have a rather strong bound.
    
    Otherwise, if $\theta_f \leq \kappa$, we apply \autoref{small_theta_cor} and observe that $\frac{1}{3}(1-\theta_f) \geq \frac{1}{3}(1-\kappa) = \frac{2}{27}$ to get
    \begin{equation*}
        N_f^s(T) = \Omega\left( T^{\frac{1}{3}(1-\theta_f) - \eps} \right) = \Omega\left( T^{\frac{2}{27} - \eps} \right) 
    \end{equation*}
    for any $\eps>0$. In either case, we obtain the desired result.
    
\end{proof}

\begin{remark}
    Under the density hypothesis for the family of twists of a fixed holomorphic form which appears in \autoref{zero_density_appendix}, so that in particular \autoref{zero_density_prop} holds with $c(\alpha) = 2$, we can take $\kappa = \frac{3}{4}$ in the preceding argument. This leads to $N_f^s(T) = \Omega\left(T^\delta\right)$ for any $\delta < \frac{1}{12}$. 
\end{remark}

\begin{remark}
    Under the generalized Lindel{\"o}f hypothesis, we can take $\mu=0$ in \autoref{subconvexity_rmk} and $\kappa = \frac{3}{4}$ in the preceding argument. This leads to $N_f^s(T) = \Omega\left(T^\delta\right)$ for any $\delta < \frac{1}{4}$.  
\end{remark}

\begin{remark}
    Under the generalized Riemann hypothesis (GRH), we can take $\mu=0$ and $\theta_f = \frac{1}{2}$ in \autoref{subconvexity_rmk}. This leads to $N_f^s(T) = \Omega\left(T^\delta\right)$ for any $\delta < \frac{1}{2}$. The weak exponent $1/2$ comes from the coarse bound in $\autoref{S_bound_lemma}$, where we apply absolute values and presumably forsake square-root cancellation on average in $S_f(\alpha+iy)$. 
    
    The better bound $N_f^s(T) \gg_{f,\eps} T(\log{T})^{-\eps}$ for any $\eps >0$ was obtained under GRH by Milinovich and Ng \cite{MN14} using the moment method. Recently, Gon\c{c}alves, de Laat, and Leijenhorst \cite{GLL23} showed under GRH that a positive proportion of the zeros have order at most two.
\end{remark}

\section{An improved estimate for \texorpdfstring{$f$}{f} of level \texorpdfstring{$1$}{1}}\label{level_one_section}

If $f$ has level $N=1$, then we will easily see that there is $\alpha \in \Q^\times$ with $\Theta(H_{f, \alpha}) \geq \theta_f$, so \autoref{initial_zero_bound_prop} gives $N_f^s(T) = \Omega\left(T^{\frac{2}{3}\theta_f -\frac{1}{6}-\eps}\right) = \Omega\left(T^{\frac{1}{6}-\eps}\right)$ for any $\eps>0$, as was proved in \cite{CG88}. We improve this result by using the sixth moment bound of Jutila \cite{Ju87} instead of subconvexity in the last step of the argument. An improvement in the exponent for the case of general level $N$ would also follow by the same reasoning, except that at present a sixth moment bound does not seem to be in the literature in such generality.

To begin, we convert pointwise values of our $L$-function into moments via the following standard lemma.

\begin{lemma}[Pointwise values to short moments]\label{pointwise_moment_lemma}
    Let $f \in S_k(\Gamma_1(N))$ be a primitive form and $T\geq 2$. For $\rho=\beta+i\gamma$ with $\beta \geq \frac{1}{2}$ and $|\gamma|\leq T$, we have 
    \begin{equation*}
        L'_f(\rho) \ll_f \log^4{T} + \log^5{T} \cdot \int_{-\log^2{T}}^{\log^2{T}} \left|L_{\overline{f}}\left(\frac{1}{2}-i(\gamma+x)\right)\right|\dd x.
    \end{equation*}
\end{lemma}
    
\begin{proof}
    Let $c = \frac{1}{100\log{T}}$. Observe that
    \begin{equation*}
        \frac{1}{2\pi i} \int_{1-i\infty}^{1+i\infty} L_f(\rho + w) \Gamma(w)^2 \dd w \ll 1,
    \end{equation*}
    as $\Re(\rho+w) \geq \frac{3}{2}$. Shifting the line of integration to $\Re(w) = \frac{1}{2}-\beta - c$, we pick up a pole at $w=0$ with residue $L'_f(\rho)$. By Stirling's formula we have the rough bound
    \begin{equation*}
        \Gamma\left(\frac{1}{2}-\beta-c+it\right) \ll e^{-|t|} \left(\left|\frac{1}{2}-\beta-c\right| + |t|\right)^{-1} \ll e^{-|t|} \left(c + |t|\right)^{-1},
    \end{equation*}
    so we get
    \begin{equation*}
        L'_f(\rho) \ll 1 + \int_{-\infty}^\infty \left| L_f\left(\frac{1}{2}-c + i(\gamma+t) \right) \right| e^{-2|t|}(c+|t|)^{-2}\dd t.
    \end{equation*}
    By convexity,
    \begin{equation*}
        \int_{\pm \frac{1}{2}\log^2{T}}^{\pm \infty} \left| L_f\left(\frac{1}{2}-c + i(\gamma+t) \right) \right| e^{-2|t|}(c+|t|)^{-2}\dd t \ll_f \int_{\frac{1}{2}\log^2{T}}^\infty e^{-t} \dd t \ll 1,
    \end{equation*}
    therefore
    \begin{equation}\label{integral_lemma_interm_eq}
        L'_f(\rho) \ll_f 1 + \log^2{T} \cdot \int_{-\frac{1}{2}\log^2{T}}^{\frac{1}{2}\log^2{T}} \left| L_f\left(\frac{1}{2}-c + i(\gamma+t) \right) \right| \dd t.
    \end{equation}
    
    The functional equation combined with Stirling's formula gives
    \begin{equation*}
        L_f\left(\frac{1}{2}-c + i(\gamma+t) \right) \ll_f \frac{\Gamma_\C\left(\frac{k}{2}+c-i(\gamma+t)\right)}{\Gamma_\C\left(\frac{k}{2}-c+i(\gamma+t)\right)} L_{\overline{f}}\left(\frac{1}{2}+ c - i(\gamma+t) \right) \ll L_{\overline{f}}\left(\frac{1}{2}+ c - i(\gamma+t) \right).
    \end{equation*}
    Now we use an argument similar to the one above. For $\vartheta = \frac{1}{2}+ c - i(\gamma+t)$ we have
    \begin{equation*}
        \frac{1}{2\pi i} \int_{1-i\infty}^{1+i\infty} L_{\overline{f}}(\vartheta + w) \Gamma(w) \dd w \ll 1,
    \end{equation*}
    and shifting the line of integration to $\Re(w) = -c$, picking up a pole at $w=0$ with residue $L_{\overline{f}}(\vartheta)$, and using $\Gamma(-c+iv) \ll e^{-|v|} (c+|v|)^{-1}$, we argue as before to get
    \begin{equation}\label{integral_lemma_final_eq}
        L_{\overline{f}}(\vartheta) \ll_f 1 + \log{T}\cdot\int_{-\frac{1}{2}\log^2{T}}^{\frac{1}{2}\log^2{T}} \left| L_{\overline{f}}\left(\frac{1}{2}-i(\gamma+t-v)\right) \right| \dd v.
    \end{equation}
    
    Inserting \eqref{integral_lemma_final_eq} into \eqref{integral_lemma_interm_eq} then gives the desired result. 
    
\end{proof}

The key new input is the lemma below.

\begin{lemma}[H\"older against sixth moment]\label{Jutila_lemma}
    Let $f \in S_k(\Gamma_0(1))$ be a primitive form and $T\geq 2$. If $\rho_n = \beta_n + i\gamma_n$ runs through the simple zeros of $\Lambda_f$ in increasing order of $|\gamma_n|$, then 
    \begin{equation*}
        \sum_{\substack{\beta_n \geq \frac{1}{2}\\ |\gamma_n| \leq T}} |L'_f(\rho_n)| \ll_{f, \eps} N_f^s(T)^\frac{5}{6} T^{\frac{1}{3}+\eps}
    \end{equation*}
    for any $\eps > 0$.
\end{lemma}

\begin{proof}
    Denote $K = T + \log^2{T}$. By \autoref{pointwise_moment_lemma}, 
    \begin{equation*}
            \sum_{\substack{\beta_n \geq \frac{1}{2}\\ |\gamma_n| \leq T}} |L'_f(\rho_n)| \ll_{f} N_f^s(T) \log^4{T} + \log^5{T} \cdot \int_{-K}^K \left|L_{\overline{f}}\left(\frac{1}{2}-it\right) \right| \cdot \sum_{\substack{\beta_n \geq \frac{1}{2}\\ |\gamma_n| \leq T}} \1_{|t-\gamma_n| \leq \log^2{T}}  \dd t. 
    \end{equation*}
    Observe that $N_f^s(x+1) - N_f^s(x) \ll_f \log(2+|x|)$ by standard zero-density results, so we have the bounds $N_f^s(T) \log^4{T} \ll_f N_f^s(T)^\frac{5}{6}T^\frac{1}{3}$ and
    \begin{equation*}
        \sum_{|\gamma_n| \leq T} \1_{|t-\gamma_n| \leq \log^2{T}} \ll_f \log^3{T}
    \end{equation*}
    for any $t \in \R$. Therefore, using H\"older's inequality twice we get
    \begin{equation*}
        \begin{split}
            &\int_{-K}^K \left|L_{\overline{f}}\left(\frac{1}{2}-it\right) \right| \cdot \sum_{\substack{\beta_n \geq \frac{1}{2}\\ |\gamma_n| \leq T}} \1_{|t-\gamma_n| \leq \log^2{T}}  \dd t \\
            &\leq \left(\int_{-K}^K \left|L_{\overline{f}}\left(\frac{1}{2}-it\right) \right|^6 \dd t\right)^\frac{1}{6} \left(\int_{-K}^K \left( \sum_{|\gamma_n| \leq T} \1_{|t-\gamma_n| \leq \log^2{T}} \right)^\frac{6}{5}  \dd t\right)^\frac{5}{6} \\
            &\ll_f \left(\int_{-K}^K \left|L_{\overline{f}}\left(\frac{1}{2}-it\right) \right|^6 \dd t\right)^\frac{1}{6}\left((\log^3{T})^\frac{1}{5} N_f^s(T) \log^2{T} \right)^\frac{5}{6}.
        \end{split}
    \end{equation*}
    
    Then Jutila's sixth moment bound \cite[Theorem 4.7]{Ju87} gives
    \begin{equation*}
        \int_{-K}^K \left|L_{\overline{f}}\left(\frac{1}{2}-it\right) \right|^6 \dd t \ll_{f, \eps} K^{2+\eps} \ll_\eps T^{2+\eps}
    \end{equation*}
    for any $\eps > 0$, and the lemma follows.
    
\end{proof}

We are now ready to obtain the desired bound for $N_f^s(T)$.

\begin{proof}[Proof of \autoref{level_1_thm}]
    For $f \in S_k(\Gamma_0(1))$ a primitive form, we can apply \eqref{first_H_diff_eq} with $p=2$ to get
    \begin{equation*}
        \begin{split}
            2^{1-2s}H_{f, 1}(s) - H_{f, \frac{1}{2}}(s) &= 2^{1-2s}\Delta_f(s) - \Delta_{f, 1, 2}(s) + R_{f, 2}(2^{-s})\Lambda_f(s) \\
            &= \left(2^{1-2s} -1 + 2P_{f, 2}(2^{-s})\right) \Delta_f(s) + R_{f, 2}(2^{-s})\Lambda_f(s) \\
            &= P_{f, 2}(2^{1-s}) \Delta_f(s) + R_{f, 2}(2^{-s})\Lambda_f(s),
        \end{split}
    \end{equation*}
    where we have used \eqref{Delta_additive_twist_decomposition_eq}. Observe that $P_{f, 2}(2^{1-s})\not=0$ and $R_{f, 2}(2^{-s})$ is holomorphic for $0 < \Re(s) < 1$, so the function above has the same poles as $\Delta_f$ in this region. We conclude that
    \begin{equation*}
        \max\left\{ \Theta(H_{f, 1}), \Theta(H_{f, \frac{1}{2}}) \right\} \geq \theta_f.
    \end{equation*}
    
    Let $\alpha=1$ or $\frac{1}{2}$ be such that $\Theta(H_{f, \alpha}) \geq \theta_f$. Also let $0 < \eps < \theta_f$ (recall that $\theta_f \geq \frac{1}{2}$ by \autoref{small_theta_cor}) and $\sigma = \theta_f - \eps$. Then since $0 < \sigma < \Theta(H_{f, \alpha})$, \autoref{S_bound_lemma} and \autoref{pole_detection_prop} give
    \begin{equation}\label{divergence_level_1_eq}
        \sum_{\substack{\rho = \beta + i\gamma \\ \text{ a pole of } \Delta_f \\ \text{with }\beta >0}} \left|\Lambda_f'(\rho)\right| e^\frac{\pi |\gamma|}{2}(1+|\gamma|)^{-\sigma-\frac{k-1}{2}} = \infty.
    \end{equation}
    By the functional equation, $\Lambda'_f(\rho) = -\rn_f N^{\frac{1}{2}-\rho} \Lambda'_{\overline{f}}(1-\rho) \ll_f \Lambda'_{\overline{f}}(1-\rho)$, so the LHS of \eqref{divergence_level_1_eq} is
    \begin{equation}\label{two_term_level_1_eq}
        \ll_f \sum_{\substack{\rho = \beta + i\gamma \\ \text{ a pole of } \Delta_f \\ \text{with }\beta \geq\frac{1}{2}}} \left|\Lambda_f'(\rho)\right| e^\frac{\pi |\gamma|}{2}(1+|\gamma|)^{-\sigma-\frac{k-1}{2}} + \sum_{\substack{\rho = \beta + i\gamma \\ \text{ a pole of } \Delta_{\overline{f}} \\ \text{with }\beta \geq\frac{1}{2}}} \left|\Lambda_{\overline{f}}'(\rho)\right| e^\frac{\pi |\gamma|}{2}(1+|\gamma|)^{-\sigma-\frac{k-1}{2}}. 
    \end{equation}
    Applying Stirling's bound $\Gamma\left(\rho + \frac{k-1}{2}\right) \ll (1+|\gamma|)^{\beta + \frac{k}{2}-1} e^{-\frac{\pi|\gamma|}{2}}$, valid for $\beta \geq \frac{1}{2}$, we have
    \begin{equation*}
        \sum_{\substack{\rho = \beta + i\gamma \\ \text{ a pole of } \Delta_f \\ \text{with }\beta \geq\frac{1}{2}}} \left|\Lambda_f'(\rho)\right| e^\frac{\pi |\gamma|}{2}(1+|\gamma|)^{-\sigma-\frac{k-1}{2}} \ll \sum_{\beta_n \geq \frac{1}{2}} \left|L_f'(\rho_n)\right| (1+|\gamma_n|)^{\beta_n-\sigma-\frac{1}{2}},
    \end{equation*}
    where we use the notation of \autoref{Jutila_lemma}. Observing that $\beta_n \leq \theta_f$ and applying \autoref{Jutila_lemma}, we obtain
    \begin{equation*}
        \begin{split}
            \sum_{\beta_n \geq \frac{1}{2}} \left|L_f'(\rho_n)\right| (1+|\gamma_n|)^{\beta_n-\sigma-\frac{1}{2}} &\ll_f 1 + \sum_{k=1}^\infty \sum_{T=2^k} T^{-\frac{1}{2}+\eps} \sum_{\substack{\beta_n \geq \frac{1}{2} \\ \frac{T}{2} < |\gamma_n| \leq T}} |L'_f(\rho_n)| \\
            & \ll_{f, \eps} 1 + \sum_{k=1}^\infty \sum_{T=2^k} N_f^s(T)^\frac{5}{6} T^{-\frac{1}{6}+2\eps}
        \end{split}
    \end{equation*}
    for any sufficiently small $\eps>0$.
    
    Now suppose by contradiction that $N_f^s(T) = o\left( T^{\frac{1}{5}-6\eps} \right)$. Then 
    \begin{equation*}
        \sum_{\beta_n \geq \frac{1}{2}} \left|L_f'(\rho_n)\right| (1+|\gamma_n|)^{\beta_n-\sigma-\frac{1}{2}} \ll_{f, \eps} 1 + \sum_{k=1}^\infty \sum_{T=2^k} o\left(T^{-3\eps}\right) < \infty.
    \end{equation*}
    The same argument, exchanging $f$ with $\overline{f}$ (and observing that $\theta_f = \theta_{\overline{f}}$), shows that the second term of \eqref{two_term_level_1_eq} is also finite. This contradicts \eqref{divergence_level_1_eq}, so we conclude that
    \begin{equation*}
        N_f^s(T) = \Omega\left( T^{\frac{1}{5}-\eps} \right)
    \end{equation*}
    for any $\eps > 0$, as desired.
    
\end{proof}


\appendix

\section{Zero-density for twists of primitive forms}\label{zero_density_appendix}

The purpose of this appendix is to obtain a zero-density estimate for character twists of a fixed form $f$ that holds in the generality required for our application and is efficient in the $Q$-aspect. We use the notation of \eqref{number_zero_def_eq} for the number of zeros in a rectangle.

\begin{proposition}[Zero-density for twists in degree two]\label{zero_density_prop}
    Let $f \in S_k(\Gamma_0(N), \xi)$ be a primitive holomorphic modular form of arbitrary weight $k$, level $N$, and nebentypus $\xi$. Then for any $Q\geq 2$, $T \geq 2$, $\eps > 0$, and $\frac{1}{2} + \eps \leq \alpha \leq 1$, there exists some $A$ depending only on $\eps$ such that
    \begin{equation*}
        \sum_{\substack{q \leq Q \\ (q, N)=1}} \sideset{}{^*}\sum_{\chi \Mod{q}} N_{f\otimes \chi}(\alpha, T) \ll_{f, \eps} \left( \left(QT\right)^{4+\eps} + \left(Q^2 T\right)^{c(\alpha)}\right)^{1-\alpha} \log^A(QT),
    \end{equation*}
    where
    \begin{equation*}
        c(\alpha) := \min\left\{\frac{3}{2-\alpha}, \frac{3}{3\alpha - 1} \right\}
    \end{equation*}
    and $\sum^*$ denotes summation over primitive characters.
\end{proposition}

The proof uses standard methods for large values of Dirichlet polynomials, and we closely follow the argument of Iwaniec-Kowalski \cite[\textsection 10.4]{IK04} for the case of Dirichlet $L$-functions, with the necessary technical modifications to deal with our case of degree two (mostly complications coming from larger conductor). This approach is based on the power of the large sieve in the $Q$-aspect. \autoref{zero_density_prop} is not particularly efficient in terms of $T$ (as we do not have access to the fourth moment in that aspect), however this is irrelevant since $T$ is fixed in our application. For $T$ small in terms of $Q$, in particular for $T$ fixed, \autoref{zero_density_prop} improves (in all ranges of $\alpha$) results of Zhang \cite{Z09} valid for $f$ of level $N=1$. 

\begin{proof}[Proof of \autoref{zero_density_prop}]
    Let $g :\R_{\geq 0} \to \R$ be given by
    \begin{equation*}
        g(x) := \kappa \int_x^\infty \exp\left(- y - \frac{1}{y}\right) \frac{dy}{y},
    \end{equation*}
    where $\kappa := (2 K_0(2))^{-1}$ is a normalizing constant so that $g(0) = 1$. Then one may check that the Mellin transform
    \begin{equation*}
        \hat{g}(z) := \int_{0}^\infty g(x) x^{z-1} \dd x
    \end{equation*}
    is odd and has a pole at $z=0$, and that $z \hat{g}(z)$ is analytic. In addition, we have the bounds
    \begin{equation}\label{g_upper_bound_eq}
        0 < g(x) < \kappa e^{-x},
    \end{equation}
    \begin{equation}\label{g_one_bound_eq}
        0 < 1-g(x) < \kappa e^{-1/x},
    \end{equation}
    and
    \begin{equation}\label{g_mellin_bound_eq}
        \hat{g}(z) \ll |z|^{|\Re(z)|-1} e^{-\frac{\pi}{2}|\Im(z)|}
    \end{equation}
    uniformly for $z \in \C$. We refer to \cite[p.\ 257-258]{IK04} for details, where one may combine Euler's reflection formula $\Gamma(z)\Gamma(1-z) = \frac{\pi}{\sin(\pi z)}$ with Stirling's formula to obtain \eqref{g_mellin_bound_eq} from \cite[(10.55)]{IK04}.
    
    Our preliminary goal is to obtain a convenient approximate functional equation for $L_{f\otimes \chi}(s)$, where from now on we assume that $s=\sigma+it$ with $\frac{1}{2}+\eps\leq \sigma \leq 1$ and $\chi$ is a primitive character modulo $q$, where $(q, N)=1$. We evaluate the sum 
    \begin{equation}\label{B_f_definition_eq}
        B_f(s, \chi) := \sum_{n=1}^\infty \lambda_f(n) \chi(n) n^{-s} g\left(\frac{n}{X}\right),
    \end{equation}
    where $X>0$ will be chosen later. By contour integration, 
    \begin{equation}\label{B_f_integral_eq}
        B_f(s, \chi) = \frac{1}{2\pi i}\int_{(1)} L_f(s+u, \chi) X^u \hat{g}(u) \dd u = L_f(s, \chi) + \frac{1}{2 \pi i} \int_{(-1)} L_f(s+u, \chi) X^u \hat{g}(u) \dd u.
    \end{equation}
    
    Since $(q, N)=1$ and $\chi \pmod{q}$ is primitive, it follows that $L_{f}(z, \chi) = L_{f\otimes \chi}(z)$ and $f \otimes \chi$ is a primitive form in $S_k(\Gamma_0(Nq^2), \xi \chi^2)$, so we have the functional equation
    \begin{equation*}
        L_f(z, \chi) = \rn_{f\otimes \chi} (Nq^2)^{\frac{1}{2}-z} \gamma_k(z) L_{\overline{f}}(1-z, \overline{\chi}),
    \end{equation*}
    where $|\rn_{f\otimes\chi}|=1$ and
    \begin{equation*}
        \gamma_k(z) := (2\pi)^{2z-1} \frac{\Gamma\left(1-z+\frac{k-1}{2}\right)}{\Gamma\left(z+\frac{k-1}{2}\right)}.
    \end{equation*}
    Using this functional equation, the integral over $\Re(u) = -1$ in \eqref{B_f_integral_eq} is equal to $-\rn_{f\otimes\chi} B_f^*(s, \chi)$, where
    \begin{equation*}
        B_f^*(s, \chi) := \frac{1}{2\pi i} \int_{(1)} (Nq^2)^{\frac{1}{2}-s+u} \gamma_k(s-u) L_{\overline{f}}(1-s+u, \overline{\chi}) X^{-u} \hat{g}(u) \dd u.
    \end{equation*}
    Expanding $L_{\overline{f}}$ into a Dirichlet series we get
    \begin{equation}\label{B_f_*_definition_eq}
        B_f^*(s, \chi) = \sum_{m=1}^\infty \lambda_{\overline{f}}(m) \overline{\chi}(m) m^{s-1} h(Xm),
    \end{equation}
    with
    \begin{equation}\label{h_definition_eq}
        h(y) := \frac{1}{2\pi i} \int_{(1)} (Nq^2)^{\frac{1}{2}-s+u} \gamma_k(s-u) y^{-u} \hat{g}(u) \dd u.
    \end{equation}
    
    In conclusion, collecting the expressions above we obtain 
    \begin{equation}\label{L_decomposition_eq}
        L_{f\otimes \chi}(s) = L_f(s, \chi) = B_f(s, \chi) + \rn_{f\otimes\chi}B_f^*(s, \chi),
    \end{equation}
    where $B_f(s, \chi)$ and $B_f^*(s, \chi)$ are given by \eqref{B_f_definition_eq} and \eqref{B_f_*_definition_eq}, respectively, and $X>0$ is arbitrary.
    
    By Euler's reflection formula and Stirling's formula, we have
    \begin{equation}\label{gamma_factor_bound_eq}
        \gamma_k(z) \ll_k |z|^{1-2\Re(z)}
    \end{equation}
    uniformly in the half-plane $\Re(z) \leq \frac{1}{2}$. Using the bounds \eqref{g_mellin_bound_eq} and \eqref{gamma_factor_bound_eq} and moving the integral \eqref{h_definition_eq} sufficiently to the right, say to the line
    \begin{equation*}
        \Re(u) = \max\left(1, \frac{1}{3} \left(\frac{y}{Nq^2|s|^2}\right)^\frac{1}{3} \right),
    \end{equation*}
    one obtains the rough uniform bound
    \begin{equation*}
        h(y) \ll_k \frac{Nq^2|s|^2}{y} \exp\left(-\frac{1}{3} \left(\frac{y}{Nq^2|s|^2}\right)^\frac{1}{3} \right).
    \end{equation*}
    Therefore, $h(mX)$ is quite small as long as $m$ is a bit larger than $Nq^2|s|^2 X^{-1}$. More precisely, by \eqref{B_f_*_definition_eq} and Deligne's bound we have
    \begin{equation}\label{B_f_*_truncated_eq}
        B_f^*(s, \chi) = \sum_{m\leq Y} \lambda_{\overline{f}}(m) \overline{\chi}(m) m^{s-1} h(mX) + O_f\left(\frac{1}{XY}\right)
    \end{equation}
    provided
    \begin{equation}\label{XY_condition_eq}
        XY \geq Nq^2|s|^2 \log^4(Nq^2|s|^2).
    \end{equation}
    
    Now we write \eqref{h_definition_eq} as
    \begin{equation*}
        h(y) := \frac{1}{2\pi i} \int_{(\eta)} (Nq^2)^{\frac{1}{2}-s+2\sigma-u} \gamma_k(s-2\sigma+u) y^{u-2\sigma} \hat{g}(2\sigma-u) \dd u
    \end{equation*}
    by changing $u$ into $2\sigma - u$ and then moving the line of integration to $\Re(u) = \eta$, where $1 < \eta < 2\sigma$. Inserting this into \eqref{B_f_*_truncated_eq} we get
    \begin{equation}\label{B_f_*_integral_eq}
        B_f^*(s, \chi) = \frac{1}{2\pi i}\int_{(\eta)} \left(\sum_{m\leq Y} \lambda_{\overline{f}}(m) \overline{\chi}(m) m^{-\overline{s}+u-1}\right) W(u) \dd u + O_f\left(\frac{1}{XY}\right),
    \end{equation}
    where
    \begin{equation*}
        W(u) := (Nq^2)^{\frac{1}{2}-s+2\sigma-u} \gamma_k(s-2\sigma+u)X^{u-2\sigma} \hat{g}(2\sigma-u).
    \end{equation*}
    Choose $\eta = 1+\eps$, which satisfies $1 < \eta < 2\sigma$ and $-\sigma+\eta \leq \frac{1}{2}$. By \eqref{g_mellin_bound_eq} and \eqref{gamma_factor_bound_eq}, for $u=\eta+iv$ we have
    \begin{equation*}
        \begin{split}
            W(u) &\ll \left(Nq^2\left(|s| + |v|\right)^2\right)^{\frac{1}{2}+\sigma-\eta} X^{\eta-2\sigma}(2\sigma-\eta)^{-1} e^{-\frac{\pi}{2}|v|} \\
            &\ll (2\sigma-\eta)^{-1} \left(\frac{Nq^2|s|^2}{X^2}\right)^{\frac{1}{2}+\sigma-\eta} X^{1-\eta} e^{-|v|}.
        \end{split}
    \end{equation*}
    Assuming that 
    \begin{equation}\label{X_squared_condition_eq}
        X^2 \geq Nq^2|s|^2,
    \end{equation}
    since $\sigma \geq \frac{1}{2}+\eps$ we get
    \begin{equation*}
        W(u) \ll \eps^{-1} X^{-\eps} e^{-|v|}.
    \end{equation*}
    
    Therefore, \eqref{B_f_*_integral_eq} becomes
    \begin{equation}\label{B_f_*_final_bound_eq}
        B_f^*(s, \chi) \ll_f \eps^{-1} X^{-\eps} \int_{-\infty}^\infty \left|\sum_{m\leq Y} \lambda_{f}(m) \chi(m) m^{-s+\eps+iv}\right| e^{-|v|} \dd v + \frac{1}{XY}.
    \end{equation}
    Denote $D := 2 \sqrt{N} Q T$ and $\L := 2\log{D}$. As a reminder, we have $s = \sigma+it$ with $\frac{1}{2}+\eps \leq \sigma \leq 1$, $\chi \pmod{q}$ primitive with $(q, N)=1$, and from now on we also assume $|t|\leq T$ and $q\leq Q$. Choose
    \begin{equation*}
        X = D \L \quad \quad \text{and} \quad \quad Y = D \L^3,
    \end{equation*}
    so that conditions \eqref{XY_condition_eq} and \eqref{X_squared_condition_eq} are satisfied. Then by \eqref{g_upper_bound_eq} the sum in \eqref{B_f_definition_eq} can be reduced to $n \leq Y$ up to an error of $O(D^{-2})$, so that combining it with \eqref{L_decomposition_eq} and \eqref{B_f_*_final_bound_eq} we get
    \begin{equation}\label{L_approx_functional_eq}
        \begin{split}
            L_{f}(s, \chi) = \sum_{n\leq Y}& \lambda_f(n) \chi(n)  n^{-s} g\left(\frac{n}{X}\right) \\
            &+ O_{f, \eps}\left(X^{-\eps} \int_{-\infty}^\infty \left| \sum_{n\leq Y} \lambda_f(n) \chi(n) n^{-s + \eps+iv} \right| e^{-|v|} \dd v + D^{-2} \right).
        \end{split}
    \end{equation}
    
    Let $1 \leq M \leq D$ and 
    \begin{equation*}
        M_f(s, \chi) := \sum_{m\leq M} b_f(m) \chi(m) m^{-s}, 
    \end{equation*}
    where the coefficients $b_f$ are inverses of $\lambda_f$ under Dirichlet convolution, i.e.\ are given by
    \begin{equation*}
        \sum_{n=1}^\infty b_f(n) n^{-s} := \prod_{p \text{ prime}} \left(1 - \lambda_f(p) p^{-s} + \xi(p)p^{-2s}\right) \quad \quad \text{for } \Re(s)>1, 
    \end{equation*}
    so that Deligne's bound implies $|b_f(n)| \leq d(n)$. From \eqref{L_approx_functional_eq} we obtain
    \begin{equation}\label{L_M_prelim_eq}
        \begin{split}
            L_{f}(s, \chi) M_f(s, \chi) = \sum_{n\leq MY}& a_f(n) \chi(n)  n^{-s} \\
            &+ O_{f, \eps}\left(\L \int_{-\infty}^\infty \left| \sum_{n\leq MY} a_f(n, v) \chi(n) n^{-s} \right| e^{-|v|} \dd v + D^{-2}M^{\frac{1}{2}} \right),
        \end{split}
    \end{equation}
    where
    \begin{equation*}
        a_f(n) := \sum_{\substack{cm=n \\ c \leq Y, m \leq M}} \lambda_f(c) g\left(\frac{c}{X}\right) b_f(m) \ll d_4(n)
    \end{equation*}
    by \eqref{g_upper_bound_eq} and similarly
    \begin{equation*}
        a_f(n, v) := \sum_{\substack{cm=n \\ c \leq Y, m \leq M}} \lambda_f(c) \left(\frac{c}{Y}\right)^{\eps+iv} b_f(m) \ll d_4(n).
    \end{equation*}
    For $n\leq M$, by \eqref{g_one_bound_eq} we have the more precise estimates
    \begin{equation*}
        a_f(n) = \sum_{cm=n} \lambda_f(c) b_f(m) \left(1 + O\left(e^{-X/c}\right) \right) = \1_{n=1} + O\left(d_4(n)D^{-2}\right)
    \end{equation*}
    and
    \begin{equation*}
        a_f(n, v) \ll \left(\frac{n}{Y}\right)^\eps \sum_{cm=n} \left|\lambda_f(c)\right| \left|b_f(m)\right| \leq \left(\frac{n}{Y}\right)^\eps d_4(n).
    \end{equation*}
    As a consequence,
    \begin{equation*}
        \left| \sum_{n\leq M} a_f(n, v) \chi(n)n^{-s} \right| \leq Y^{-\eps} \sum_{n\leq M} d_4(n) n^{-\frac{1}{2}} \ll Y^{-\eps} M^\frac{1}{2} \log^3(2M).
    \end{equation*}
    We want this to be $O_\eps(\L^{-2})$, which holds assuming for instance
    \begin{equation}\label{M_condition_eq}
        M \leq D^{\eps}.
    \end{equation}
    
    In that case, using the bounds above in \eqref{L_M_prelim_eq} gives
    \begin{equation}\label{L_M_cleaner_eq}
        \begin{split}
            L_{f}(s, \chi) M_f(s, \chi) = 1 &+ \sum_{M < n\leq MY} a_f(n) \chi(n)  n^{-s} \\
            &+ O_{f, \eps}\left(\L \int_{-\infty}^\infty \left| \sum_{M < n\leq MY} a_f(n, v) \chi(n) n^{-s} \right| e^{-|v|} \dd v + \L^{-1} \right).
        \end{split}
    \end{equation}
    To unify the treatment of the sum and integral, we consider the measure
    \begin{equation*}
        d\mu := \frac{1}{3}e^{-|v|}dv + \frac{1}{3} \delta(v),
    \end{equation*}
    where $dv$ denotes the Lebesgue measure on $\R$, $\delta(v)$ is the point measure at $v=0$, and the factor $\frac{1}{3}$ is a normalization that makes $\int_{-\infty}^\infty d\mu = 1$. Then \eqref{L_M_cleaner_eq} can be written as
    \begin{equation}\label{L_M_final_eq}
        L_{f}(s, \chi) M_f(s, \chi) - 1 \ll_{f, \eps} \L \int_{-\infty}^\infty \left| \sum_{M < n\leq MY} a_f(n, v) \chi(n) n^{-s} \right| \dd \mu(v) + \L^{-1}
    \end{equation}
    after redefining $a_f(n, 0) := a_f(n)$. For convenience, we also redefine $a_f(n, v) := 0$ for $n \leq M$ or $n > MY$. From now on the only properties about the coefficients we will use are that they do not depend on $s$ or $\chi$ and satisfy $a_f(n, v) \ll d_4(n)$.
    
    Now, assume that $\rho$ is a zero of $L_{f\otimes\chi}(s) = L_f(s, \chi)$ for some primitive $\chi \pmod{q}$ with $(q, N)=1$, $q\leq Q$, $\frac{1}{2}+\eps \leq \alpha \leq \Re(\rho) \leq 1$, and $|\Im(\rho)|\leq T$. If $D$ is sufficiently large in terms of $f$ and $\eps$ (which we can assume, otherwise \autoref{zero_density_prop} follows trivially), then \eqref{L_M_final_eq} implies
    \begin{equation*}
        \int_{-\infty}^\infty \left| \sum_{M < n\leq MY} a_f(n, v) \chi(n) n^{-\rho} \right| \dd \mu(v) \gg_{f, \eps} \L^{-1}.
    \end{equation*}
    We break the summation into dyadic segments $J < n \leq 2J$ for $J :=2^\ell M$, $0 \leq \ell \leq L := \lfloor \log{Y} / \log{2} \rfloor \ll \L$. Denote
    \begin{equation*}
        D_\ell(s, \chi) := \int_{-\infty}^\infty \left| \sum_{J < n\leq 2J} a_f(n, v) \chi(n) n^{-s} \right| \dd \mu(v).
    \end{equation*}
    Then for each such zero $\rho$ being counted there exists some $\ell$ such that
    \begin{equation}\label{D_lower_bound_eq}
        D_\ell(\rho, \chi) \gg_{f, \eps} \L^{-2}.
    \end{equation}
    If $\S_\ell$ denotes the set of relevant pairs $(\rho, \chi)$ satisfying \eqref{D_lower_bound_eq} and $R_\ell := |\S_\ell|$, then the total number $R$ of zeros being counted in \autoref{zero_density_prop} satisfies
    \begin{equation*}
        R \leq \sum_{\ell = 0}^L R_\ell \ll \L \max_{0 \leq \ell \leq L} R_\ell.
    \end{equation*}
    
    Raising $D_\ell(s, \chi)$ to a suitable power $2r$, for $r \geq 2$ (depending on $J$), we get
    \begin{equation*}
        D_\ell(s, \chi)^{2r} \leq \int_{-\infty}^\infty \left| \sum_{J < n\leq 2J} a_f(n, v) \chi(n) n^{-s} \right|^{2r} \dd \mu(v) =: \int_{-\infty}^\infty \left| \sum_{P < n\leq 2^r P} c_f(n, v) \chi(n) n^{-s} \right|^2 \dd \mu(v),
    \end{equation*}
    where $P := J^r$ falls in the segment
    \begin{equation}\label{P_condition_eq}
        Z \leq P \leq (MY)^2 + Z^\frac{3}{2}
    \end{equation}
    for $Z$ that will be chosen later satisfying 
    \begin{equation}\label{Z_condition_eq}
        MY \leq Z \ll D^{100}.
    \end{equation}
   Observe that an integer $r \geq 2$ such that \eqref{P_condition_eq} holds exists. From now on we choose
    \begin{equation*}
        M = D^\frac{\eps}{4},
    \end{equation*}
    so that $r$ is bounded in terms of $\eps$, by \eqref{Z_condition_eq}. Observe that condition \eqref{M_condition_eq} is automatically satisfied.
    
    By \eqref{D_lower_bound_eq}, we conclude that
    \begin{equation}\label{final_dirichlet_poly_eq}
        R_\ell \ll_{f, \eps} \L^{4r} \int_{-\infty}^\infty \sum_{(\rho, \chi) \in \S_\ell} \left| \sum_{P < n\leq 2^r P} c_f(n, v) \chi(n) n^{-\rho} \right|^2 \dd \mu(v).
    \end{equation}
    The coefficients satisfy $c_f(n, v) \ll_r d_{4r}(n)$, as $a_f(n, v) \ll d_4(n)$, so 
    \begin{equation}\label{coefficient_bound_eq}
        \sum_{P < n \leq 2^r P} |c_f(n, v)|^2 n^{-2\alpha} \leq P^{1-2\alpha} \L^B
    \end{equation}
    for some $B$ depending only on $r$ (and therefore $\eps$). We can now apply results about large values of Dirichlet polynomials to the integrand of \eqref{final_dirichlet_poly_eq}, after separating the zeros $\rho$ for each given $\chi$ into $O(\L)$ families of $1$-spaced points. Let $H:= Q^2T$.
    
    Suppose that $\frac{1}{2} + \eps \leq \alpha \leq \frac{3}{4}$. By \eqref{coefficient_bound_eq} and the large sieve inequality \cite[Theorem 9.13]{IK04}, we have
    \begin{equation*}
        \begin{split}
            R_\ell &\ll_{f, \eps} (P + H) P^{1-2\alpha} \L^C \\
            &\ll \left((MY)^{4(1-\alpha)} + Z^{3(1-\alpha)} + H Z^{1-2\alpha} \right) \L^C
        \end{split}
    \end{equation*}
    for some $C$ depending only on $\eps$, where we have used \eqref{P_condition_eq}. If $H \leq (MY)^{3-2\alpha}$, choose $Z = MY$, which trivially satisfies \eqref{Z_condition_eq}, so
    \begin{equation*}
        R_\ell \ll_{f, \eps} (MY)^{4(1-\alpha)} \L^C \leq D^{(4+\eps)(1-\alpha)} \L^{C+6}    
    \end{equation*}
    and we are done. If instead $H \geq (MY)^{3-2\alpha}$, then we can make the optimal choice $Z = H^{\frac{1}{2-\alpha}}$ and \eqref{Z_condition_eq} is satisfied, so we get
    \begin{equation*}
        R_\ell \ll_{f, \eps} \left((MY)^{4(1-\alpha)} + H^{\frac{3(1-\alpha)}{2-\alpha}} \right) \L^C \leq  \left(D^{(4+\eps)(1-\alpha)}+ H^{\frac{3(1-\alpha)}{2-\alpha}} \right) \L^{C+6}
    \end{equation*}
    as desired.
    
    Finally, suppose that $\frac{3}{4} \leq \alpha \leq 1$. By the Hal{\'a}sz-Montgomery-Huxley method \cite[Theorem 9.15]{IK04}, we have
    \begin{equation*}
        \begin{split}
            R_\ell \ll_{f, \eps} \left(P + R_\ell^\frac{2}{3} H^\frac{1}{3} P^\frac{1}{3} \right) P^{1-2\alpha} \L^C
        \end{split}
    \end{equation*}
    for some $C$ depending only on $\eps$, which implies
    \begin{equation*}
        \begin{split}
            R_\ell &\ll_{f, \eps} \left(P^{2-2\alpha} + H P^{4-6\alpha} \right) \L^{3C} \\
            &\ll \left((MY)^{4(1-\alpha)} + Z^{3(1-\alpha)} + H Z^{4-6\alpha} \right) \L^{3C},
        \end{split}
    \end{equation*}
    where again we have used \eqref{P_condition_eq}. If $H \leq (MY)^{2\alpha}$, we choose $Z = MY$, which trivially satisfies \eqref{Z_condition_eq}, and get
    \begin{equation*}
        R_\ell \ll_{f, \eps} (MY)^{4(1-\alpha)} \L^{3C} \leq D^{(4+\eps)(1-\alpha)} \L^{3C+3},    
    \end{equation*}
    so we are done. If instead $H \geq (MY)^{2\alpha}$, then we make the optimal choice $Z = H^\frac{1}{3\alpha-1}$, which in this case satisfies \eqref{Z_condition_eq}. Therefore, 
    \begin{equation*}
        \begin{split}
            R_\ell \ll_{f, \eps} \left( (MY)^{4(1-\alpha)} + H^\frac{3(1-\alpha)}{3\alpha-1} \right) \L^{3C} \leq  \left(D^{(4+\eps)(1-\alpha)}+ H^\frac{3(1-\alpha)}{3\alpha-1} \right) \L^{3C+3}
        \end{split}
    \end{equation*}
    as desired.
    
\end{proof}


\nocite{*}  
\bibliographystyle{abbrv}
\bibliography{references}

\end{document}